\documentclass[a4paper,11pt]{amsart}
\usepackage[left=2.7cm,right=2.7cm,top=3.5cm,bottom=3cm]{geometry}

\usepackage{amsthm,amssymb,amsmath,amsfonts,mathrsfs,amscd}
\usepackage[latin1]{inputenc}
\usepackage[all]{xy}
\usepackage{latexsym}
\usepackage{longtable}
\usepackage{url}
\usepackage{stmaryrd}

\newfont{\cyr}{wncyr10 scaled 1100}
\newfont{\cyrr}{wncyr9 scaled 1000}

\theoremstyle{plain}
\newtheorem{theorem}{Theorem}[section]
\newtheorem{proposition}[theorem]{Proposition}
\newtheorem{lemma}[theorem]{Lemma}
\newtheorem{corollary}[theorem]{Corollary}

\theoremstyle{definition}

\newtheorem{definition}[theorem]{Definition}

\theoremstyle{remark}
\newtheorem{remark}[theorem]{Remark}

\newcommand{\Q}{\mathbb Q}

\newcommand{\Z}{\mathbb Z}

\newcommand{\C}{\mathbb C}

\DeclareMathOperator{\Spec}{Spec}

\DeclareMathOperator{\End}{End}

\DeclareMathOperator{\Emb}{Emb}

\DeclareMathOperator{\Hom}{Hom}

\DeclareMathOperator{\Gal}{Gal}
\DeclareMathOperator{\GL}{GL}
\DeclareMathOperator{\PGL}{PGL}
\DeclareMathOperator{\SL}{SL}

\DeclareMathOperator{\M}{M}
\DeclareMathOperator{\CH}{CH}

\DeclareMathOperator{\AJ}{AJ}
\DeclareMathOperator{\Ext}{Ext}

\newcommand{\ord}{\mathrm{ord}}

\newcommand{\unr}{\mathrm{unr}}

\usepackage[usenames]{color}
\definecolor{Indigo}{rgb}{0.2,0.1,0.7}
\definecolor{Violet}{rgb}{0.5,0.1,0.7}
\definecolor{White}{rgb}{1,1,1}
\definecolor{Green}{rgb}{0.1,0.9,0.2}

\newcommand{\longmono}{\mbox{$\lhook\joinrel\longrightarrow$}}

\newcommand{\mat}[4]{\left(\begin{array}{cc}#1&#2\\#3&#4\end{array}\right)}
\newcommand{\smallmat}[4]{\bigl(\begin{smallmatrix}#1&#2\\#3&#4\end{smallmatrix}\bigr)}

\newcommand{\cO}{{\mathcal O}}

\newcommand{\HH}{{\mathcal H}}

\newcommand{\PP}{\mathbb P}

\newcommand{\cl}{c\ell}

\include{thebibliography}

\begin{document}
\title[]{Generalized Heegner cycles on Mumford curves}
\author{Matteo Longo and Maria Rosaria Pati}
\maketitle
\begin{abstract}
We study generalised Heegner cycles, originally introduced by Bertolini-Darmon-Prasanna for modular curves in \cite{BDP}, in the context of Mumford curves. 
The main result of this paper relates generalized Heegner cycles with the two variable anticyclotomic $p$-adic $L$-function attached to a Coleman family $f_\infty$ and an imaginary quadratic field $K$, constructed in \cite{BD} and \cite{Sev}. While in \cite{BD} and \cite{Sev} only the restriction to the central critical line of this 2 variable $p$-adic $L$-function is considered, our generalised Heegner cycles allow us to study the restriction of this function to non-central critical lines. The main result expresses the derivative along the weight variable of this anticyclotomic $p$-adic $L$-function restricted to non necesserely central critical lines as a combination of the image of generalized Heegner cycles under a $p$-adic Abel-Jacobi map. In studying generalised Heegner cycles in the context of Mumford curves, we also obtain an  extension of a result of Masdeu \cite{Masdeu} for the (one variable) anticyclotomic $p$-adic $L$-function of a modular form $f$ and $K$ at non-central critical integers.  
\end{abstract}

\tableofcontents
\section{Introduction}

Generalized Heegner cycles have been introduced by Bertolini-Darmon-Prasanna in \cite{BDP} with the aim of studying certain anticyclotomic $p$-adic $L$-functions of modular forms of level $\Gamma_1(N)$, where $p\nmid N$ is a prime number, twisted by Hecke characters of an imaginary quadratic field $K/\Q$ in which all primes dividing $N$ are split. These cycles are defined by means of the cohomology of the motive $(\mathcal{E}^n\times E^n,\epsilon)$, where $\mathcal{E}^n$ is a smooth compactification of the $n$-fold product of the universal elliptic curve $\mathcal{E}\rightarrow X_1(N)$, $E$ is an auxiliary elliptic curve with CM by $\mathcal{O}_K$ and $\epsilon$ is a suitable projector in the ring of rational correspondences on $\mathcal{E}^n\times E^n$. The work \cite{BDP} has been generalised by Brooks \cite{Brooks} to the case when $X_1(N)$ is replaced by a Shimura curve, and therefore $\mathcal{E}$ and $E$ are replaced by a universal false elliptic curve $\mathcal{A}$ and a false elliptic curve $A$ with CM by $\mathcal{O}_K$; in \cite{Brooks}, $N$ is assumed to be prime to $p$ as in \cite{BDP}, and one allows factorisations of $N$ into coprime integers $N=N^+\cdot N^-$ where all primes dividing $N^+$ are split in $K$, all primes dividing $N^-$ are inert in $K$, and $N^-$ is the square-free product of an even number of distinct prime factors.  

Along a different direction Masdeu in \cite{Masdeu} has defined generalized Heegner cycles for Mumford curves; in this setting we fix a modular form $f$ of weight $k$ and level $\Gamma_0(N)$, an imaginary quadratic field $K$ and a factorisation $N=p\cdot N^+\cdot N^-$ into coprime factors so that $p$ is a prime number, all primes dividing $N^+$ are split in $K$, all primes dividing $pN^-$ are inert in $K$, and $N^-$ is the square-free product of an odd number of distinct prime factors. The fiber at $p$ of Shimura curves attached to quaternion algebras $\mathcal{B}$ of discriminant $pN^-$ can be described by means of Mumford curves, \emph{i.e.} quotients $\Gamma\backslash\mathcal{H}_p$ of the $p$-adic upper half plane $\mathcal{H}_p=\C_p-\Q_p$ by an arithmetic group $\Gamma\subseteq B^\times$, where $B$ is the definite quaternion algebra of discriminant $N^-$ obtained from $\mathcal{B}$ by interchanging the invariants $\infty$ and $p$. In this case, generalized Heegner cycles are constructed by means of the cohomology of $(\mathcal{A}^\frac{n}{2}\times E^n,\epsilon_M)$, where $\mathcal{A}$ is the universal false elliptic curve as in \cite{Brooks}, $E$ is a fixed elliptic curve with CM by $\mathcal{O}_K$ as in \cite{BDP}, $n:=k-2$ and $\epsilon_M$ is a suitable projector on $\mathcal{A}^\frac{n}{2}\times E^n$. The main result of \cite{Masdeu} expresses the derivative of the anticyclotomic $p$-adic $L$-function attached to $f$ and $K$ at integers $j$ in the critical strip $1\leq j\leq k-1$ as linear combinations of the images of generalised Heegner cycles via the $p$-adic Abel-Jacobi map, evaluated at suitable differential forms. The main tools used in \cite{Masdeu} is the analysis by Iovita-Spiess \cite{IS} of the realisations (\'etale and de Rham) of the motive $(\mathcal{A}^{\frac{n}{2}},\epsilon_\mathcal{A})$. 

This paper continues the work initiated by \cite{Masdeu} in the context of Mumford curves, but instead of the motive $(\mathcal{A}^\frac{n}{2}\times E^n,\epsilon_M)$ considered in \cite{Masdeu} we study the motive $(\mathcal{A}^\frac{n}{2}\times A^\frac{n}{2},\epsilon)$ with $\mathcal{A}$ a universal false elliptic curve over a Shimura curve, $A$ a fixed false elliptic curve with CM by $\mathcal{O}_K$ and $\epsilon$ a projector in $\mathrm{Corr}_X(\mathcal{A}^\frac{n}{2}\times A^\frac{n}{2})$. Here the setting is the same as in \cite{Masdeu}: we fix a modular form $f$ of level $\Gamma_0(N)$, an imaginary quadratic field $K$ and a factorisation $N=p\cdot N^+\cdot N^-$ into coprime factors so that $p$ is a prime number, all primes dividing $N^+$ are split in $K$, all primes dividing $pN^-$ are inert in $K$, and $N^-$ is the square-free product of an odd number of distinct prime factors. It turns out that our motive seems to be more flexible than the motive considered in \cite{Masdeu}, and more natural because both the universal abelian variety and the fixed abelian variety with CM are false elliptic curves.
In this context we define generalised Heegner cycles, and we study them using techniques from \cite{IS} and \cite{BDP}. 

Our main results investigate the relation between generalised Heegner cycles and anticyclotomic $p$-adic $L$-functions, 
especially in the context of $p$-adic variation of modular forms. Fix an imaginary quadratic field $K$ and a modular form $f$ of weight $k_0$ and level $\Gamma_0(N)$, with $N=pN^+N^-$ as above, having finite slope at $p$. Let $f_\infty$ be the Coleman family of modular forms passing through $f$. 
In the ordinary case with $k_0=2$, Bertolini and Darmon introduced in \cite{BD} a $p$-adic $L$-function in the weight-variable $k$ interpolating special values at central critical points of anticyclotomic $L$-functions of newforms $f_k^\sharp$ whose $p$-stabilisations are the classical specialisations $f_k$ of $f_\infty$. In particular, this $p$-adic $L$-function is non-zero and vanishes at $k=2$. When $f$ corresponds to an elliptic curve, the main result of \cite{BD} expresses the first derivative along the weight variable $k$ of this anticyclotomic $p$-adic $L$-function valued at the point $k=2$ as linear combination of Heegner points. This results has been extended by Seveso \cite{Sev} in the finite slope case and $k_0\geq 2$ by expressing the first derivative along the weight variable $k$ of this anticyclotomic $p$-adic $L$-function at $k=k_0$ as linear combination of Heegner cycles. 

The $p$-adic $L$-functions studied in \cite{BD} and \cite{Sev} are restriction to the central critical line $s=k/2$ of a
$p$-adic $L$-function $\mathcal{L}_p(s,k)$ in two $p$-adic variables $s$ and $k$; in light of the results of \cite{Masdeu}, it is then natural to investigate the restriction $\mathcal{L}_p^{(j)}(k)=\mathcal{L}_p(k,k+j)$ of these $p$-adic $L$-functions along directions $s=k/2+j$, with $-k/2<j<k/2$ an integer in the critical strip. In the spirit of \cite{BD} and \cite{Sev}, for each $j$ such that 
$\mathcal{L}_p^{(j)}(k_0)=0$, we show that 
the derivative $\frac{d}{dk}\mathcal{L}_p^{(j)}(k)$ at $k=k_0$ can be expressed 
as linear combinations of our generalised Heegner cycles. 

We now state our main result in a more precise form. 
Let $f\in S_{k_0}(\Gamma_0(N))$ be a newform of weight $k_0$ and level $\Gamma_0(N)$, $K$ an imaginary quadratic field, $N=p\cdot N^+\cdot N^-$ a factorisation of $N$ into coprime integers such that $p$ is a prime, all prime factors dividing $N^+$ (respectively, $pN^-$) are split (respectively, inert) in $K$, and $N^-$ is a square-free product of an odd number of primes. 
Let $\mathcal{B}/\Q$ be the indefinite quaternion algebra of discriminant $pN^-$, and $\Gamma\subseteq B^\times$ the arithmetic subgroup corresponding to the choice of an Eichler order of level $N^+$ in the definite quaternion algebra $B/\Q$ of discriminant $N^-$. Let $X$ be the Shimura curve of level $\Gamma$. After choosing an auxiliary prime integer $M\ge 5$ prime to $N$ and a $\Gamma_1(M)$-level structure $\Gamma_M\subseteq\Gamma$, consider the Shimura curve $X_M\rightarrow X$ of level $\Gamma_M$ and the universal false elliptic curve $\mathcal{A}\rightarrow X_M$. Fix a false elliptic curve $A_0$ with CM by $\mathcal{O}_K$. For any isogeny $\varphi:A_0\rightarrow A$ we construct a generalised Heegner cycle $\Delta_\varphi$ in the Chow group $\CH^{n_0+1}(\mathcal{D})$ of the Chow motive $\mathcal{D}:=(\mathcal{A}^\frac{n_0}{2}\times A_0^\frac{n_0}{2},\epsilon)$, 
where $n_0=k_0-2$.  
For any positive even integer $k$, let $M_{k}(\Gamma)$ be the $\C_p$-vector space of rigid analytic quaternionic modular forms of weight $k$ and level $\Gamma$; elements of $M_{k}(\Gamma)$ are functions from $\mathcal{H}_p=\C_p-\Q_p$ to $\C_p$ which transform under the action of $\Gamma$ by the automorphic factor of weight $k$. In particular, the Jacquet-Langlands correspondence allows us to see $f$ as an element of $M_{k_0}(\Gamma)$.  
Let $V_{n_0}$ denote the dual of the $\C_p$-vector space $\mathcal{P}_{n_0}$ of polynomials in one variable of degree at most $n_0$. We construct a $p$-adic Abel-Jacobi map 
\[\AJ_p:\CH^{n_0+1}(\mathcal{D})\longrightarrow (M_{k_0}(\Gamma)\otimes V_{n_0})^\vee\] where the target denotes 
the $\C_p$-linear dual of $M_{k_0}(\Gamma)\otimes V_{n_0}$.
On the other hand, denote \[\mathcal{W}=\Hom_\mathrm{cont}(\Z_p^\times,\Q_p^\times)\] the weight space, and view $\Z\subseteq\mathcal{W}$ by the map $k\mapsto[x\mapsto x^{k-2}]$.
For any integer $j$ with $-k_0/2<j<k_0/2$, we construct a function 
$k\mapsto \mathcal{L}_p^{(j)}(k)$ defined in a sufficiently small connected neighborhood of $U$ of $k_0\in\mathcal{W}$.
When $j\equiv0\pmod{p+1}$, $\mathcal{L}_p^{(j)}(k)$ coincides with  the restriction to the line $s=k/2+j$ of the two variable $p$-adic $L$-function of \cite{BD}, \cite{Sev}; thus in particular the value of this function at $j=0$ correspond to the one variable $p$-adic $L$-function studied in \cite{BD}, \cite{Sev}.
The notation used below to denote this function is more involved, 
but in the introduction we prefer to keep the notational complexity at minimum
stating our main result, Theorem \ref{theoremA}, in the case when the class number of $K$ is equal to $1$:  
see Definitions \ref{def of p-adic L-function 1} and \ref{def of p-adic L-function 2}
for the complete notation, keeping in mind that if the class number 
of $K$ is $1$ then the two functions in 
Definitions  \ref{def of p-adic L-function 1} and \ref{def of p-adic L-function 2}
are the same, and $\chi$ in \emph{loc. cit.} is trivial.
Thus, our main result, for which as remarked above we assume that the class number of $K$ is one to simplify the statement, is the following:

\begin{theorem}\label{theoremA} For integers $-k_0/2<j<k_0/2$
with $j\equiv0\pmod{p+1}$
we have \[\mathcal{L}_p^{(j)}(k_0)=0\] and 
there exists an isogeny $\varphi:A_0\rightarrow A$ and are elements $v_\varphi^{(j)}$ and $\bar{v}_\varphi^{(j)}$ in $V_{n_0}$ such that we have 
\[
\left(\frac{d}{dk}\mathcal{L}_p^{(j)}(k)\right)_{|k=k_0}=c_\varphi\left(\AJ_p(\Delta_{\varphi})(f\otimes v_\varphi^{(j)})+
\omega_p\AJ_p(\Delta_{\bar{\varphi}})(f\otimes \bar{v}_\varphi^{(j)})\right).\]
\end{theorem} 

In the theorem above, $c_\varphi\in\bar{\Q}_p^\times$  is an explicit constant which only depends on $\varphi$, $\omega_p\in\{\pm1\}$ is the eigenvalue of the Atkin-Lehner involution acting on $f$, 
and if $\varphi:A_0\rightarrow A$ is an isogeny, we 
denote by $\bar{\varphi}\colon A_0\rightarrow\bar{A}$ the isogeny obtained by $\varphi$ composing with the generator of $\Gal(K/\Q)$ (recall that $A$ is defined over $K$ by the theory of complex multiplication, under the assumption that $K$ has class number one). 
This result is a special case of Theorem \ref{TheoremB} below, which also considers twists by certain anticyclotomic characters of $K$, and holds for arbitrary class number of $K$.  

In studying our generalized Heegner cycles, we also obtain a second result similar in spirit to that of \cite{Masdeu}, expressing the first derivative of the anticyclotomic $p$-adic $L$-function attached to $f$ and $K$ at integers $j$ in the critical strip $1\leq j\leq k_0-1$ as linear combinations of our generalized Heegner cycles, valued at suitable differential forms; although the result is similar in spirit to that of \cite{Masdeu}, it has a different shape, due to the different motives used, and furthermore generalises that of \cite{Masdeu} to certain anticyclotomic characters. 
We state a simplified version (again for trivial characters and class number of $K$ equal to $1$) of this result, referring to 
Theorem \ref{TheoremA} and the comments following it for the notation.  

\begin{theorem} Let $L_p(f/K,s)$ be the anticyclotomic $p$-adic $L$-function attached to $f$ and $K$ and $-k_0/2< j< k_0/2$ an integer with $j\equiv0\pmod{p+1}$. 
For each such $j$ we have $L_p(f/K,k_0/2+j)=0$ and  there exists an isogeny $\varphi:A_0\rightarrow A$
such that 
\[L_p'\left(f/K,k/2+j\right)=
d_\varphi\cdot\left(\AJ_p(\Delta_{\varphi})(f\otimes v_{{\varphi}}^{(j)})-\omega_p\cdot
\AJ_p(\Delta_{\bar{{\varphi}}})(f\otimes \bar{v}_{{\varphi}}^{(j)})
\right)
.\]\end{theorem}
Here $d_\varphi\in\bar{\Q}_p^\times$ is an explicit constant which only depends on $\varphi$. 
This result is a special case of Theorem \ref{TheoremA} below, which, as for Theorem \ref{theoremA}, 
also considers twists by certain anticyclotomic characters of $K$, and holds for arbitrary class number of $K$.  

\section{Shimura curves}
\label{section: Shimura curves}

In this section we collect come preliminaries on Shimura curves which will be needed in this paper. We fix an integer $N$ with a coprime factorization $N=pN^+N^-$ such that $p\nmid N^+N^-$ is a prime number, and $N^-$ is a square free product of an odd number of primes factors. 

\subsection{Shimura curves}
Let $\mathcal{B}$ be the indefinite quaternion algebra over $\Q$ of discriminant $pN^-$. Fix a maximal order $\mathcal{R}^\mathrm{max}$ in $\mathcal{B}$ and an Eichler order $\mathcal{R}$ of level $N^+$ contained in $\mathcal{R}^\mathrm{max}$. The Shimura curve $X=X_{N^+,pN^-}/\Q$ is the coarse moduli scheme representing the functor which takes a $\Q$-scheme $S$ to 
isomorphism classes of abelian surfaces with quaternonic multiplication by $\mathcal{R}^\mathrm{max}$ and level $N^+$-structure, \emph{i.e.} triples $(A,\iota,C)$ where
\begin{enumerate}
\item $A$ is an abelian surface over a $\Q$-scheme $S$;
\item $\iota\colon \mathcal{R}^\mathrm{max}\rightarrow\End_S(A)$ is an inclusion defining an $\mathcal{R}^\mathrm{max}$-module structure on $A/S$;
\item $C\subset A$ is a subsgroup scheme locally isomorphic to $\Z/N^+\Z$, stable and locally cyclic under the action of $\mathcal{R}$.
\end{enumerate}
The scheme $X$ is a smooth, projective and geometrically connected curve over $\Q$. A triple $(A,\iota,C)$  is called a \emph{false elliptic curve with level $N^+$-structure}, and the abelian surface $A$ is 
called a \emph{false elliptic curve}. An isogeny $\varphi:A\rightarrow A'$ of false elliptic curves is said to be a \emph{false isogeny} if it commutes with the action of $\mathcal{R}^\mathrm{max}$.  

The fiber at $p$ of $X$ is a Mumford curve, as we will review now.  
Let $\mathcal{H}_p$ denote the rigid analytic space over $\Q_p$ whose points over field extensions $L/\Q_p$ are given by $\mathcal{H}_p(L)=L-\Q_p$ 
(see for example \cite[Chapter 5]{DarmonBook} or \cite[Section 1]{DasgTeit}, where the rigid analytic structure of $\mathcal{H}_p$ is also carefully described).  
Let $B/\Q$ be the definite quaternion algebra over $\Q$ of discriminant $N^-$ and let $R$ be an Eichler $\Z[\frac{1}{p}]$-order of level $N^+$ in $B$. By fixing an isomorphism $\iota_p\colon B_p\rightarrow \M_2(\Q_p)$ the group $\Gamma$ of elements of reduced norm $1$ in $R$ can be identified with a discrete subgroup of $\SL_2(\Q_p)$. We let $\SL_2(\Q_p)$ act on $\mathcal{H}_p(L)$, for each field extension $L/\Q_p$, by fractional linear transformations $z\mapsto\frac{az+b}{cz+d}$ for $\smallmat abcd\in\SL_2(\Q_p)$ and $z\in\mathcal{H}_p(L)$. We may then form the quotient $X_{\Gamma,\Q_p}=\Gamma\backslash\HH_p$ 
and for any field extension $F/\Q_p$, its base change $X_{\Gamma,F}=X_{\Gamma,\Q_p}\otimes_{\Q_p}F$, which, when $F$ contains $\Q_{p^2}$, is a Mumford curve defined over $F$ (in general, it is a twist of a Mumford curve by a quadratic character).  
The Cerednik-Drinfeld theorem states the existence of an isomorphism
\begin{equation}\label{cer-drin}
X_{\Q_{p^2}}\simeq X_{\Gamma,\Q_{p^2}}
\end{equation}
of algebraic curves defined over $\Q_{p^2}$.
See \cite[Section 4]{JL}, 
\cite[Theorem 1.3]{BD-Heegner} or \cite[Chapitre III]{BC} for details. We put 
$X_\Gamma=X_{\Gamma,\hat{\Q}_{p}^\unr}$ to simplify the notation, 
where $\hat{\Q}_p^\unr$ is the completion of the maximal unramified extension $\Q_p^\unr$ 
of $\Q_p$. 

\subsection{An auxiliary fine moduli problem}\label{sec:2.2}
Fix $M\geq 3$ an integer relatively prime to $N$. Let $X_M$ be the fine moduli scheme representing abelian surfaces with quaternionic multiplication by $\mathcal{R}^\mathrm{max}$, level $N^+$-structure and full level $M$-structure over $\Q$-schemes, \emph{i.e.} 
quadruples $(A,\iota,C,\overline{\nu})$ where
\begin{enumerate}
\item $(A,\iota,C)$ is an abelian surface with quaternionic multiplication by $\mathcal{R}^\mathrm{max}$ and level $N^+$-structure over a $\Q$-scheme $S$;
\item $\overline{\nu}\colon (\mathcal{R}^\mathrm{max}/M\mathcal{R}^\mathrm{max})_S\rightarrow A[M]$ is a $\mathcal{R}^\mathrm{max}$-equivariant isomorphism from the constant group scheme $(\mathcal{R}^\mathrm{max}/M\mathcal{R}^\mathrm{max})$ to the group scheme of $M$-division points of $A$.
\end{enumerate}
Quadruplets $(A,\iota,C,\overline{\nu})$ are called \emph{false elliptic curves with level $(N^+,\nu)$-structure}. 
The scheme $X_M$ is a smooth projective curve over $\Q$ which is not geometrically connected. The morphism $X_M\rightarrow X$ given by forgetting the level $M$-structure is a Galois covering with Galois group isomorphic to $G_M/\{\pm 1\}$, where 
\[G_M:=(\mathcal{R}^\mathrm{max}/M\mathcal{R}^\mathrm{max})^\times.\] 
We denote $\mathcal{A}\rightarrow X_M$ the universal abelian surface. 

Over $\hat{\Q}_p^\unr$, the curve $X_M$ decomposes as a disjoint union of Mumford curves
\begin{equation}\label{split}
X_{M,\hat{\Q}_p^\mathrm{unr}}=X_M\otimes_{\Q}\hat{\Q}_p^\unr\simeq \coprod_{(\Z/M\Z)^\times}\Gamma_M\backslash \HH_p^\unr,\end{equation}
for a suitable congruence subgroup 
$\Gamma_M\subset\SL_2(\Q_p)$, where we write $\mathcal{H}_p^\unr=\mathcal{H}_p\otimes_{\Q_p}\hat{\Q}_{p}^\unr$ (this isomorphism can be realised over any extension of $\Q_{p^2}$ containing all $\varphi(M)$-roots of unity, where $\varphi(M)=\sharp(\Z/M\Z)^\times$). 
See \cite[Section 5]{IS} for more details. 

\subsection{Modular forms}\label{sec:modular forms} We introduce in this subsection
two definitions of modular forms on quaternion algebras. 
\begin{definition}\label{definition modular forms}
Let $F$ be a field of characteristic zero and $k\geq 2$ an even integer. A $F$-valued modular form of weight $k$ on $X$ is a global section of the sheaf $(\Omega_{X_F/F}^1)^{\otimes\frac{k}{2}}$. We denote the space of these modular forms by $M_{k}(X,F)$. 
\end{definition}

The Jacquet-Langlands correspondence implies the existence of an isomorphism of $K$-vector spaces 
\[M_{k}(X,F)\simeq S_{k}(\Gamma_0(N),F)^{pN^-\text{-new}}\] where the right hand side denotes the $F$-vector space of $F$-valued cusp forms of weight $k$ and level $\Gamma_0(N)$ which are new at the primes dividing $pN^-$. This isomorphism is compatible with the action of the Hecke operators and Atkin-Lehner involutions, defined on both sides. For details, see \cite[Theorem 1.2]{BD-Heegner}. 

\begin{definition} Let $F$ be a field of characteristic zero and $k\geq 2$ an even integer.  
A $p$-adic modular form of weight $k$ for $\Gamma$ defined over $F$ is a rigid analytic function $f$ on $\HH_p$ defined over $F$ satisfying the rule 
$$f(\gamma z)=(cz+d)^k f(z)\qquad\text{for all}\ \gamma=
\begin{pmatrix}
a & b\\
c & d
\end{pmatrix}\in\Gamma \text{ and }z\in\mathcal{H}_p(\C_p)=\C_p-\Q_p.$$
The space of these $p$-adic modular forms will be denoted by $M_k(\Gamma,F)$ 
and for $F=\hat{\Q}_p^\unr$ we set $M_k(\Gamma)=M_k(\Gamma,\hat{\Q}_p^\unr)$. 
\end{definition}
Using the Cerednik-Drinfeld isomorphism \eqref{cer-drin}, 
one easily shows that the map 
$f\mapsto f(z)dz^{\otimes\frac{k}{2}}$ 
establishes and isomorphism between $M_k(\Gamma,F)$ and $M_k(X,F)$ for all fields $F$ containing $\Q_{p^2}$. 
 
\section{The generalised Kuga-Sato motive}\label{section: motive}
Let $N=pN^+N^-$ be fixed as in \S \ref{section: Shimura curves}. 
Let $k\geq4$ be an even integer and put $n=k-2$ and $m=n/2$.
Fix a quadratic imaginary field $K$ satisfying the following assumption: all primes dividing $N^+$ (respectively, $pN^-$) are split (respectively, inert) in $K$.  

\subsection{Definition} 
We begin by recalling some generalities on Chow motives, following mainly \cite[\S5]{IS}. 
Let $F$ be a field of characteristic zero, and $S$ a smooth quasi-projective connected variety over $F$. We denote $\mathcal{M}(S)$ the category of effective relative Chow motives over $S$ with respect to graded correspondences (\cite[\S 1.3]{DM}, \cite[Sec. 1]{Scholl-motives}). We will only 
use motives of the form $(X,\epsilon)=(X,\epsilon,0)$ where $X$ is a smooth projective $S$-scheme and $\epsilon\in\mathrm{Corr}^0_S(X,X)$ is a projector (\emph{i.e.} $\epsilon\circ\epsilon=\epsilon$) in the ring of correspondences on $X$ of degree $0$ (see \cite[\S1.2, \S1.3]{Scholl-motives} for details). 
If $S=\Spec(F)$, write $\mathcal{M}(F)=\mathcal{M}\left(\Spec(F)\right)$. Denote $R_p:\mathcal{M}(S)\rightarrow D^b(S,\Q_p)$ the $p$-adic realisation functor to the bounded derived category $D^b(S,\Q_p)$ of $\Q_p$-sheaves over $S$ (\cite[\S1.8]{DM}), thus for $\mathcal{M}=(X,\epsilon)\in\mathcal{M}(S)$, $R_p(\mathcal{M})$ denotes the $p$-adic realization of $\mathcal{M}$ as a motive over $S$. We can also consider $\mathcal{M}$ as a Chow motive over $F$ by applying the canonical functor $\mathcal{M}(S)\rightarrow\mathcal{M}(F)$, and if $\mathcal{M}=(X,\epsilon)$ for an abelian scheme $\pi:X\rightarrow S$, the $p$-adic realization $H_p(\mathcal{M})$ of $\mathcal{M}$ as a motive over $F$ is given by 
\[H_p^r(\mathcal{M})=H^r(\bar{S},R_p(\mathcal{M}))=\epsilon_*\cdot \bigoplus_{i+j=r}H^i(\bar{S},R^j\pi_*\Q_p),\] 
where $\bar{S}=S\otimes_F\bar{F}$
(see \cite[Proposition 5.9]{Besser} for the argument). 

We denote 
\begin{equation}\label{cycle map}
\cl_{\mathcal{M}}^{(i)}:\CH^i(\mathcal{M})\longrightarrow H^{2i}_p(\mathcal{M})\end{equation}
 the cycle class map (\cite[(40)]{IS}), whose kernel 
is denoted $\CH(\mathcal{M})_0$ (this map will not be used until
Section \ref{section: Abel-Jacobi}, but we prefer to introduce it here to collect all notations concerning Chow motives; the same applies to  
\eqref{Gamma functor} and \eqref{cycle class} below). 

Let $F$ be an unramified field extension of $\Q_p$. 
For a semistable representation of $G_F=\Gal(\bar{F}/F)$, let $D_{\mathrm{st},F}$ denote 
the semistable Dieudonn\'e functor over $F$ (see \cite[\S 2]{IS}); 
so if $V$ is a semistable representation of $G_F$, then 
$D_{\mathrm{st},F}(V)$ is a filtered Frobenius monodromy module over $F$ (see \cite[\S 2]{IS}); the category of such objects is denoted $\mathrm{MF}_F^{\phi,N}$, and for an object $D$ in this category we denote $F^\bullet(D)$ its filtration. 
For an object $D$ in $\mathrm{MF}_F^{\phi,N}$, define 
\begin{equation}\label{Gamma functor}
\Gamma(D)=\Hom_{\mathrm{MF}_F^{\phi,N}}(F,D)=\Ext^0_{\mathrm{MF}_F^{\phi,N}}(F,D)=
F^0(D)\cap D^{\phi=\mathrm{id}, N=0}.\end{equation}
Here $\Hom_{\mathrm{MF}_F^{\phi,N}}(\cdot,\cdot)$ denotes 
homomorphisms in the category $\mathrm{MF}_F^{\phi,N}$,
$\phi$ is the Frobenius morphism, 
$\mathrm{id}$ is the identity morphism,  
and $N$ is the monodromy operator of the object $D$. 
In particular, if the $p$-adic realization $H_p(\mathcal{M})$ of $\mathcal{M}$ is semistable, then 
the cycle class map $\cl_\mathcal{M}^{(i)}$ takes the form (\cite[(47)]{IS})  
\begin{equation}\label{cycle class}
\cl_\mathcal{M}^{(i)}:\CH^i(\mathcal{M})\longrightarrow\Gamma\left(D_{\mathrm{st},F}(H^{2i}_p(\mathcal{M})(F))\right).\end{equation}

Let $A_0$ be a fixed abelian surface with quaternionic multiplication and full level-$M$ structure, defined over $H$ (the Hilbert class field of $K$) and with complex multiplication by $\cO_K$; the action of $\mathcal{O}_K$ is required to commute with the quaternionic action, and this implies that $A_0$ is isogenous to $E\times E$ for an elliptic curve $E$ with CM by $\mathcal{O}_K$.  Fix a field $F\supset H$ and consider the $(2n+1)$-dimensional variety $Y_m$ over $F$ given by
\[Y_m:=\mathcal{A}^m\times A_0^m.\]
Here $\mathcal{A}^m$ is the $m$-fold fiber product of $\mathcal{A}$ over $X_M$. 
The variety $Y_m$ is equipped with a proper morphism $\pi\colon Y_m\rightarrow X_M$ with $2n$-dimensional fibers: the fibers above points $x$ of $X_M$ are products of the form $A_x^m\times A^m_0$, where $A_x$ is the fiber of $\mathcal{A}\rightarrow X_M$ at $x$. 

Denote $\epsilon_\mathcal{A}$ the projector in \cite[Appendix 10.1]{IS}; this is an idempotent in the ring of correspondences $\mathrm{Corr}_{X_M}(\mathcal{A}^m,\mathcal{A}^m)$. The projector $\epsilon_\mathcal{A}$ defines then a projector $\epsilon_{A_0}$.  
One can then define the motive
\[\mathcal{D}_M:=(Y_m,\epsilon_M)\] defined over $F$, 
where $\epsilon_M=(\epsilon_\mathcal{A},\epsilon_{A_0})$. 
In the previous notation, $\mathcal{D}_M\in\mathcal{M}(X_M)$.  

We now descent $\mathcal{D}_M$ to a motive over the Shimura curve $X$. Observe that the group 
\[G_M:=(\mathcal{R}^\mathrm{max}/M\mathcal{R}^\mathrm{max})^\times\simeq \GL_2(\Z/M\Z)\] acts as $X$-automorphism on $X_M$ and $\mathcal{A}^m$. It follows that the element
$p_G:=\frac{1}{|G_M|}\sum_{g\in G_M}g$
can be seen as a projector in $\mathrm{Corr}_X(Y_m,Y_m)$, which acts trivially on $A_0^m$. Since it commutes with $\epsilon_M$ (viewed as projector in $\mathrm{Corr}_X(Y_m,Y_m)$), their product $\epsilon=p_G\cdot\epsilon_M$ is a projector, and we can define a new motive $\mathcal{D}$ over $X$, the \emph{generalised Kuga-Sato motive}, as
$$\mathcal{D}:=(Y_m,\epsilon).$$
In the previous notation, $\mathcal{D}\in\mathcal{M}(X)$. 
We also denote $\mathcal{M}_M=(\mathcal{A}^m,\epsilon_\mathcal{A})$ the motive in $\mathcal{M}(X_M)$ considered in \cite{IS}, and 
we write $\mathcal{M}_{A_0}=(A^m_0,\epsilon_{A_0})$, also in $\mathcal{M}(X_M)$; 
then $\mathcal{D}_M=\mathcal{M}_M\otimes\mathcal{M}_{A_0}$.  Moreover, if we write $\mathcal{M}=(\mathcal{A}^m,p_{G}\cdot\epsilon_\mathcal{A})$ then we have 
\begin{equation}\label{decomposition of the motive} 
\mathcal{D}=\mathcal{M}\otimes\mathcal{M}_{A_0}\end{equation}
viewing $\mathcal{M}_{A_0}$ as a motive over $X$ (recall that the tensor product on the category of Chow motives is induced by the fiber product \cite[page 203]{DM}). 
Finally, note that $H_p^{2n+1}(\mathcal{D})$ is equipped with a structure of $G_F=\Gal(\bar F/F)$-representation.

\subsection{The \'etale realization} 
We consider now the sheaf $\mathbb{L}_n$ over $X_M$ introduced in \cite[Section 5]{IS}, which is defined as follows. First, define $\mathbb{L}_2$ as the intersection of the kernels of the maps $b-\mathrm{Nr}(b):R^2\pi_*\Q_p\rightarrow R^2\pi_*\Q_p$, as $b$ varies in $\mathcal{B}$, where $\mathrm{Nr}$ denote the reduced norm map; 
next, for any integer $n>2$, consider the non-degenerate 
pairing 
$R^2\pi_*\Q_p\otimes  R^2\pi_*\Q_p\rightarrow\Q_p(-2)$ given by cup product and the Laplace operator 
$\Delta_m:\mathrm{Sym}^m(\mathbb{L}_2)\rightarrow\left(\mathrm{Sym}^{m-2}(\mathbb{L}_2)\right)(-2)$ associated with this pairing, and 
define $\mathbb{L}_n$ to be the kernel of $\Delta_m$. 

Let $x_{A_0}$ be the closed point of $X_M$ corresponding to the abelian surface $A_0$ and $\bar{x}_{A_0}=x_{A_0}\otimes_F \bar{\Q}$.

\begin{proposition} \label{prop3.1}
The $p$-adic realization $H_p(\mathcal{D})$ of $\mathcal{D}$ is different from zero in degree $2n+1$ only, and we have
$$H_p^{2n+1}(\mathcal{D})=H^1(\overline{X}_M,\mathbb{L}_n)^{G_M}\otimes(\mathbb{L}_n)_{\bar{x}_{A_0}}.$$
\end{proposition} 

\begin{proof} The $p$-adic realization $R_p(\mathcal{M}_M)$ of the motive $\mathcal{M}_M$ over $X_M$ is 
$\mathbb{L}_n[-n]$ (\cite[(71)]{IS}); by \cite[ Lemma 10.1]{IS} 
the $p$-adic realization $H_p(\mathcal{M})$ of $\mathcal{M}$ is concentrated 
in degree $n+1$ and we have \[H_p^{n+1}(\mathcal{M})\simeq H^1(\overline{X}_M,\mathbb{L}_n)^{G_M}.\]  On the other hand, the $p$-adic realization 
$R_p(\mathcal{M}_{A_0})$ of the 
motive $\mathcal{M}_{A_0}$ over $X_M$ is the fiber at $x_{A_0}$ of $R_p(\mathcal{M}_M)=\mathbb{L}_n[-n]$ (\cite[(71)]{IS}); therefore, $H_p(\mathcal{M}_{A_0})=H^*\left(\overline{X}_M,(\mathbb{L}_n[-n])_{x_{A_0}}\right)$. Since $H^i\left(\overline{X}_M,(\mathbb{L}_n)_{x_{A_0}}\right)=0$ for $i\neq 0$, we see that 
$H_p^n(\mathcal{M}_{A_0})=(\mathbb{L}_n)_{\bar{x}_{A_0}}$ and $H_p^i(\mathcal{M}_{A_0})=0$ for $i\neq n$.  
The Kunneth formula (\cite[\S 1.8]{DM}) implies the result.\end{proof}
\begin{remark}
Considered as a $G_{\hat{\Q}_p^\mathrm{unr}}$-representation $H_p(\mathcal{D})$ is semistable since the category of semistable representations is an abelian tensor category.
\end{remark}

\subsection{The de Rham realisation} \label{sec: de Rham realisation} 
Let $V_n:=\mathcal{P}_n^\vee$ be the dual of the vector space of polynomials of degree $\leq n$ with coefficients in 
$\Q_{p}$ equipped with the left $\GL_2$-action given by
\[(A\cdot R)(P(X))=R(P(X)\cdot A)\] for all 
$A=\smallmat abcd$, where the right action of $A$ on 
a polynomial 
$P(X)\in\mathcal{P}_n$ is via the formula 
$P(X)\cdot A=(cX+d)^n P\left(\frac{aX+b}{cX+d}\right)$. 
The $\Q_p$-vector space $V_n$ is equipped with a symmetric bilinear form 
\begin{equation}\label{pairing V}
\langle,\rangle_{V_n}:V_n\otimes V_n\longrightarrow \det{}^{\otimes{n}}\end{equation}
in $\mathrm{Rep}_{\Q_p}(\GL_2)$, the category of
$\Q_p$-representations of $\GL_2$, defined as follows. First, we define $\langle,\rangle_{V_2}$ 
for $n=2$. Let $\mathrm{ad}^0=\{U\in\M_2|\mathrm{trace}(U)=0\}$, where $\mathrm{trace}:\M_2(\Q_p)\rightarrow\Q_p$ is the trace map; $\mathrm{ad}^0$ is equipped with a right $\GL_2$-action by 
$U\cdot A=\overline{A}\cdot U\cdot A$ for $U\in\mathrm{ad}^0$ and $A\in\GL_2$, where 
for $A=\smallmat abcd$, we put $\overline{A}=\smallmat {d}{-b}{-c}{a}$. 
The map $\mathrm{ad}^0\rightarrow \mathcal{P}_2$ which takes $U$ to 
\begin{equation}\label{polynomial} 
P_U(X)=\mathrm{trace}\left(U\cdot\mat X{-X^2}{1}{-X}\right)\end{equation} is an isomorphism of 
right $\GL_2$-modules.  
For $P_{U_1},P_{U_2}\in\mathcal{P}_2$, 
we define a pairing on $\mathcal{P}_2$ by $\langle P_{U_1},P_{U_2}\rangle_{\mathcal{P}_2}=\langle U_1,U_2\rangle_{V_2}=-\mathrm{trace}(U_1\cdot \overline{U}_2)$. This defines the pairing on $V_2$ by duality. More generally, 
we define a pairing $\langle,\rangle_{V_n}$ 
on $\mathrm{Sym}^{n/2}(\mathrm{ad}^0)$ by the formula 
\[\langle u_1\cdots u_{n/2},v_1\cdots v_{n/2}\rangle_{V_n}=\frac{1}{(n/2)!}\sum_{\sigma\in\Sigma_{n/2}}\langle u_1,v_{\sigma(1)}\rangle_{V_2}\cdot \langle u_n,v_{\sigma(n/2)}\rangle_{V_2}.\] The 
map $\mathrm{Sym}^{n/2}(\mathrm{ad}^0)\rightarrow\mathcal{P}_n$ induced by $U\mapsto P_U(X)$ gives by duality a map $V_n\rightarrow \mathrm{Sym}^{n/2}(\mathrm{ad}^0)$, 
and we obtain a pairing on $V_n$, which we also denote $\langle,\rangle_{V_n}$, 
from that on $\mathrm{Sym}^{n/2}(\mathrm{ad}^0)$. 

We consider the $\GL_2\times\GL_2$-representation
$V_n\{m\}=V_n\odot\det^{\otimes{m}}$ (see \cite[page 345]{IS} for the notation) and let $\mathcal{V}_n=\mathcal{E}(V_n\{m\})$ denote the filtered $F$-isocrystal on $\hat{\mathcal{H}}_p^\unr$ associated with $V_n\{m\}$; 
here $\hat{\mathcal{H}}_p$ is the formal model of $\mathcal{H}_p$ over $\Z_p$, and $\hat{\mathcal{H}}_p^\unr$ its base change to $\hat{\Z}_p^\unr$, the valuation ring of $\hat{\Q}_p^\unr$. 
See \cite[page 346]{IS} and \cite[page 1024]{Masdeu} for more details on this definition. 
Define the sheaf of $\mathcal{O}_{X_\Gamma}$-modules $\mathcal{V}_{n,n}=\mathcal{V}_n\otimes(\mathcal{V}_n)_{z_{A_0}}$ where $z_{A_0}$ is a point in $\HH_p(\Q_{p^2})$ such that $\Gamma z_{A_0}$ corresponds to the abelian surface $A_0$. Then  
\begin{equation}\label{iso1}
H^1_\mathrm{dR}(X_\Gamma,\mathcal{V}_{n,n})=
H^1_\mathrm{dR}(X_\Gamma,\mathcal{V}_{n})\otimes _{\hat{\Q}_p^\unr}(\mathcal{V}_n)_{z_{A_0}}.\end{equation}

The vector space $H^1_\mathrm{dR}(X_\Gamma,\mathcal{V}_{n,n})$ has a stucture of filtered Frobenius monodromy module in $\mathrm{MF}_{\hat{\Q}_p^\unr}^{\phi,N}$. 

\begin{proposition}\label{prop3.2}
$D_{\mathrm{st},\hat{\Q}_p^\unr}\left(H_p^{2n+1}(\mathcal{D})\right)\simeq
H^1_\mathrm{dR}(X_\Gamma,\mathcal{V}_{n,n})$ as filtered Frobenius monodromy modules in 
$\mathrm{MF}_{\hat{\Q}_p^\unr}^{\phi,N}$.  
\end{proposition}

\begin{proof} By \cite[Theorem 5.9]{IS} we have 
\[D_{\mathrm{st},\hat{\Q}_p^\unr}\left(H_p^{n+1}(\mathcal{\mathcal{M}})\right)\simeq D_{\mathrm{st},\hat{\Q}_p^\unr}\left(H^1(\overline{X}_M,\mathbb{L}_n)^{G_M}\right)\simeq
H^1_\mathrm{dR}(X_{\Gamma},\mathcal{V}_n)\]
and, by \cite[Remark 5.14]{IS} we have $D_{\mathrm{st},\hat{\Q}_p^\unr}\left((\mathbb{L}_n)_{\bar{x}_{A_0}}\right)\simeq (\mathcal{V}_n)_{z_{A_0}}$. The result follows from Proposition \ref{prop3.1}, equation 
\eqref{iso1} and the compatibility of the functor $D_\mathrm{st}$ with tensor products (see \cite[pg. 145]{BrCo}). 
\end{proof}

We now describe of 
$D_{\mathrm{st},\hat{\Q}_p^\unr}\left(H_p^{2n+1}(\mathcal{D})\right)$ as 
filtered Frobenius monodromy module.

We begin with the filtration. 
For $i=0,\dots,n$ and $z\in \mathcal{H}_p(\hat{\Q}_p^\unr)$, define $\partial^i\in (\mathcal{V}_n)_{z}\simeq V_n$ 
by \[\partial^i(P(X))=\left(\frac{d^i}{dX^i}P(X)\right)_{X=z}\]
for $P(X)\in\mathcal{P}_n$. Let $\Q_p\cdot\partial^i$ be the 
$\Q_p$-subspace of $V_n$ 
generated by $\partial^i$. 
The $i$-th step of the filtration of $V_n$ is given by 
\[F^i(V_n)=
\begin{cases}
\text{$V_n$ if $i\leq 0$} \\
\text{$\sum_{j=0}^{n-i}\Q_p\cdot \partial^j$ for $0\leq i\leq n$}\\
\text{$0$ if $i\geq n+1$}. 
\end{cases}
\] 
The $i$-th step of 
the filtration of $H^1_\mathrm{dR}(X_\Gamma,\mathcal{V}_n)$ is 
\begin{equation}\label{de Rham filtration} 
F^i\left(H^1_\mathrm{dR}(X_\Gamma,\mathcal{V}_n)\right)=
\begin{cases} \text{$H^1_\mathrm{dR}(X_\Gamma,\mathcal{V}_n)$ if $i\leq 0$}\\
\text{$M_k(\Gamma)$ if $1\leq i\leq n+1$}\\
\text{$0$ is $i\geq n+2$}
\end{cases}
\end{equation} See \cite[Proposition 6.1]{IS} for proofs. In particular, the isomorphism $M_k(\Gamma)\rightarrow F^{n+1}(H^1_\mathrm{dR}(X_\Gamma,\mathcal{V}_n))$ is given by
\begin{equation}\label{omega_f}
f\longmapsto\omega_f:=f(z)\partial^0\otimes dz.
\end{equation}
From \eqref{de Rham filtration} and Proposition \ref{prop3.2} we see that 
the $(n+1)$-step of the filtration of 
$D_{\mathrm{st},\hat{\Q}_{p}^\unr}\left(H_p^{2n+1}(\mathcal{D})\right)$ is 
\begin{equation}\label{filtration}
F^{n+1}\left(D_{\mathrm{st},\hat{\Q}_{p}^\unr}\left(H_p^{2n+1}(\mathcal{D})\right)\right)
=M_k(\Gamma)\otimes (\mathcal{V}_n)_{z_{A_0}} .\end{equation}

We also need an explicit description of the monodromy operator 
on $D_{\mathrm{st},\hat{\Q}_{p}^\unr}\left(H_p^{2n+1}(\mathcal{D})\right)$. 
We first describe the monodromy operator on $H^1_\mathrm{dR}(X_\Gamma,\mathcal{V}_n)$.
Let $\mathcal{T}$ 
denote the Bruhat-Tits tree of $\PGL_2(\Q_p)$, and denote $\overrightarrow{\mathcal{E}}$ and $\mathcal{V}$ the set of oriented edges and vertices of $\mathcal{T}$, respectively. If $e=(v_1,v_2)\in\overrightarrow{\mathcal{E}}$, we denote by $\overline{e}$ the oriented edge $(v_2,v_1)$. Let $C^0\left((V_n)_{\hat{\Q}_p^\unr}\right)$ be the set of maps $\mathcal{V}\rightarrow (V_n)_{\hat{\Q}_p^\unr}$ 
and 
$C^1\left((V_n)_{\hat{\Q}_p^\unr}\right)$ be the set of maps 
$\overrightarrow{\mathcal{E}}\rightarrow (V_n)_{\hat{\Q}_p^\unr}$ such that $f(\overline{e})=-f(e)$ for all $e\in\overrightarrow{\mathcal{E}}$, 
where $(V_n)_{\hat{\Q}_p^\unr}=V_n\otimes_{\Q_p}\hat{\Q}_p^\unr$. 
The group $\Gamma$ acts on $f\in C^i\left((V_n)_{\hat{\Q}_p^\mathrm{unr}}\right)$ by $\gamma(f)=\gamma\circ f\circ \gamma^{-1}$. Let 
\[\epsilon\colon C^1\left((V_n)_{\hat{\Q}_p^\mathrm{unr}}\right)^\Gamma\longrightarrow H^1\left(\Gamma,(V_n)_{\hat{\mathbb{Q}}_p^\mathrm{unr}}\right)\] be the connecting homomorphism arising from the short exact sequence
$$0\longrightarrow (V_n)_{\hat{\Q}_p^\mathrm{unr}}\longrightarrow C^0\left((V_n)_{\hat{\Q}_p^\mathrm{unr}}\right)\overset{\delta}{\longrightarrow} C^1\left((V_n)_{\hat{\Q}_p^\mathrm{unr}})\right)\longrightarrow 0,$$
where $\delta$ is the homomorphism defined by $\delta(f)(e)=f(v_1)-f(v_2)$ for $e=(v_1,v_2)$.
The map $\epsilon$ induces the following isomorphism that we also denote by $\epsilon$
$$\epsilon\colon C^1\left((V_n)_{\hat{\Q}_p^\mathrm{unr}}\right)^\Gamma \big/ C^0\left((V_n)_{\hat{\Q}_p^\mathrm{unr}}\right)^\Gamma\longrightarrow H^1\left(\Gamma,(V_n)_{\hat{\Q}_p^\mathrm{unr}}\right).$$
Let $A_e\subset\HH_p$ be the oriented annulus in $\HH_p$ corresponding to $e$ and $U_v\subset\HH_p$ be the affinoid corresponding to $v\in\mathcal{V}$, which are obtained as inverse images of the reduction map (see \cite[page 342]{IS}). 
Recall that $H^1_\mathrm{dR}(X_\Gamma,\mathcal{V}_n)$ can be identified with the $\hat{\Q}_p^\unr$-vector space of $V_n$-valued, $\Gamma$-invariant differential forms of the second kind on $\HH_p$ modulo exact forms (\cite[page 348]{IS}).  
Let $\omega$ be a $V_n$-valued $\Gamma$-invariant  differential of the second kind on $\HH_p$. We define $I(\omega)$ to be the map which assigns to an oriented edge $e\in\overrightarrow{\mathcal{E}}$ the value $I(\omega)(e)=\mathrm{Res}_e(\omega)$, where $\mathrm{Res}_e$ denotes the annular residue along $A_e$. 
If $\omega$ is exact, $I(\omega)=0$. Thus $I$ gives a well-defined map
\begin{equation}\label{I}
I\colon H^1_\mathrm{dR}\left(X_\Gamma,\mathcal{V}_n\right)\longrightarrow C^1\left((V_n)_{\hat{\Q}_p^\mathrm{unr}}\right)^\Gamma.\end{equation}
The $\hat{\Q}_p^\unr$-vector space $\bigoplus_{v\in\mathcal{V}}H^0_\mathrm{dR}(U_v,\mathcal{V}_n)$ can be identified with $C^0\left((V_n)_{\hat{\Q}_p^\mathrm{unr}}\right)$, and the subspace of $\bigoplus_{e\in\overrightarrow{\mathcal{E}}}H^0_\mathrm{dR}(A_e,\mathcal{V}_n)$ consisting of elements $\{f_e\}_{e\in\overrightarrow{\mathcal{E}}}$ such that $f_{\overline{e}}=-f_e$ can be identified with $C^1\left((V_n)_{\hat{\Q}_p^\mathrm{unr}}\right)$.
Since the set $\{U_v\}_{v\in\mathcal{V}}$ is an admissible covering of $\HH_p$, the Mayer-Vietoris sequence yields an embedding
$$C^1\left((V_n)_{\hat{\Q}_p^\mathrm{unr}}\right)^\Gamma\big/C^0\left((V_n)_{\hat{\Q}_p^\mathrm{unr}}\right)^\Gamma\longmono H^1_\mathrm{dR}(X_\Gamma,\mathcal{V}_n).$$
Precomposing with $\epsilon$, we obtain an embedding
\begin{equation}\label{iota}
\iota\colon H^1\left(\Gamma,(V_n)_{\hat{\Q}_p^\mathrm{unr}}\right)\longmono H^1_\mathrm{dR}(X_\Gamma,\mathcal{V}_n)\end{equation}
This map admits a natural left inverse
\begin{equation}\label{P}
P\colon H^1_\mathrm{dR}(X_\Gamma,\mathcal{V}_n)\longrightarrow H^1\left(\Gamma,(V_n)_{\hat{\Q}_p^\mathrm{unr}}\right),\end{equation} which takes  
$\omega$ to the class of the cocycle $\gamma\mapsto\gamma(F_\omega)-F_\omega$. Here $F_\omega$ is a primitive of $\omega$ in the sense of Coleman, i.e. $dF_\omega=\omega$ (see \cite[Lemma 4.4]{Col}).
 
Define now the monodromy operator $N_n$ on $H^1_\mathrm{dR}(X_\Gamma,\mathcal{V}_n)$ as the composite $\iota\circ(-\epsilon)\circ I$.
On the other hand, the monodromy operator $N_{(\mathcal{V}_n)_{z_{A_0}}}$ on the filtered $(\phi,N)$-module $(\mathcal{V}_n)_{z_{A_0}}$ is trivial. Therefore, since $D_{\mathrm{st},\hat{\Q}_p^\mathrm{unr}}(H^{2n+1}_p(\mathcal{D}))$ is isomorphic to $H^1_\mathrm{dR}(X_\Gamma,\mathcal{V}_n)\otimes (\mathcal{V}_n)_{z_{A_0}}$ in the category of filtered Frobenius monodromy modules, its monodromy operator is given by
\begin{equation}\label{monodromy}
N=\mathrm{id}_n\otimes N_{(\mathcal{V}_n)_{z_{A_0}}}+N_n\otimes\mathrm{id}_{(\mathcal{V}_n)_{z_{A_0}}}=N_n\otimes\mathrm{id}_{(\mathcal{V}_n)_{z_{A_0}}},\end{equation}
where $\mathrm{id}_\bullet$ denote identity operators. 

We now describe the Frobenius operator on $D_{\mathrm{st},\hat{\Q}_p^\mathrm{unr}}(H^{2n+1}_p(\mathcal{D}))$. First, 
$H^1\left(\Gamma,(V_n)_{\hat{\Q}_p^\mathrm{unr}}\right)$ has a Frobenius endomorphism induced by the map $p^\frac{n}{2}\otimes \sigma$ on $(V_n)_{\hat{\Q}_p^\mathrm{unr}}=V_n\otimes_{\Q_p}\hat{\Q}_p^\mathrm{unr}$, where $\sigma$ denotes the absolute Frobenius automorphism on $\hat{\Q}_p^\mathrm{unr}$.
As defined in \cite[Section 4]{IS}, $\Phi_n$ is the unique operator on $H^1_\mathrm{dR}(X_\Gamma,\mathcal{V}_n)$ satisying $N_n\Phi_n=p\Phi_n N_n$ and which is compatible (with respect to $\iota$ and $P$) with the Frobenius on $H^1\left(\Gamma,(V_n)_{\hat{\Q}_p^\mathrm{unr}}\right)$. On the other hand, the Frobenius on the filtered $(\phi,N)$-module $(\mathcal{V}_n)_{z_{A_0}}$ is given by $\Phi_{(\mathcal{V}_n)_{z_{A_0}}}=p^{\frac{n}{2}}\otimes\sigma$ acting on the underlying vector space $V_n\otimes_{\Q_p}\hat{\Q}_p^\mathrm{unr}$.
The Frobenius operator on $D_{\mathrm{st},\hat{\Q}_p^\mathrm{unr}}(H^{2n+1}_p(\mathcal{D}))$ is given by
$$\Phi=\Phi_n\otimes \Phi_{(\mathcal{V}_n)_{z_{A_0}}}.$$
Note that $N$ and $\Phi$ satisfy the relation $N\Phi=p\Phi N$.

Recall that $H^1_\mathrm{dR}(X_\Gamma,\mathcal{V}_n)$ is equipped with a non-degenerate pairing 
\[\langle,\rangle_{\mathcal{V}_n}:H^1_\mathrm{dR}(X_\Gamma,\mathcal{V}_n)\otimes
H^1_\mathrm{dR}(X_\Gamma,\mathcal{V}_n)\longrightarrow \hat{\Q}_{p}^\unr[n+1]\]
in $\mathrm{MF}_{\hat{\Q}_p^\mathrm{unr}}^{\phi,N}$, 
which is induced from $\langle,\rangle_{V_n}$; 
see \cite[\S 5]{IS}, especially \cite[Remark 5.12]{IS}, for definitions and details. 
Let 
\begin{equation}\label{pairing V_n}
\langle,\rangle_{\mathcal{V}_{n,n}}:
H^1_\mathrm{dR}(X_\Gamma,\mathcal{V}_{n,n})
\otimes 
H^1_\mathrm{dR}(X_\Gamma,\mathcal{V}_{n,n})\longrightarrow 
\hat{\Q}_{p}^\unr[n+1]\otimes\det{}^{\otimes{n}}
\end{equation} be the induced symmetric non-degenerate pairing defined by $\langle,\rangle_{\mathcal{V}_{n,n}}=\langle,\rangle_{\mathcal{V}_n}\otimes\langle,\rangle_{V_n}$ (where we also use the isomorphisms 
$(\mathcal{V}_n)_{z_{A_0}}\simeq V_n$ to define a pairing on $(\mathcal{V}_n)_{z_{A_0}}$ via that on 
$V_n$). 
If we denote $V^\vee$ the $F$-linear dual of a $F$-vector space $V$, from \eqref{filtration} 
and the non-degeneracy of $\langle,\rangle_{\mathcal{V}_{n,n}}$ we obtain an isomorphism of $\hat{\Q}_p^\unr$-vector spaces: 
\begin{equation}\label{F n+1}
\frac{D_{\mathrm{st},\hat{\Q}_p^\unr}\left(H_p^{2n+1}(\mathcal{D})\right)}{F^{n+1}\left(D_{\mathrm{st},\hat{\Q}_p^\unr}\left(H_p^{2n+1}(\mathcal{D})\right)\right)}\simeq 
\left(F^{n+1}\left(D_{\mathrm{st},\hat{\Q}_p^\unr}\left(H_p^{2n+1}(\mathcal{D})\right)\right)\right)^\vee
\simeq\left(M_k(\Gamma)\otimes (\mathcal{V}_n)_{z_{A_0}}\right)^\vee.\end{equation}
 
\section{$p$-adic Abel-Jacobi maps}\label{section: Abel-Jacobi}
Let the notation be as in Section \ref{section: motive}: $N=pN^+N^-$ is a factorisation of the integer $N\geq 1$ into coprime integers $p,N^+,N^-$ with $p\nmid N^+N^-$ a prime number, and $N^-$ be a square-free product of an odd number of factors;   
$k\geq 4$ is an even integer and put $n=k-2$ and $m=n/2$; 
$K$ is a quadratic imaginary field such that all primes dividing $N^+$ (respectively, $pN^-$) are split (respectively, inert) in $K$. 

\subsection{Definition of the Abel-Jacobi map}
Let $[\Delta]$ be the class of a null-homologous cycle $\Delta$ of codimension $n+1$ in $\CH^{n+1}(\mathcal{D})(F)$, where $F\subseteq\bar{\Q}$ is a field containing the Hilbert class field of $K$; here 
\[\CH^{n+1}(\mathcal{D})(F)=\epsilon\cdot\CH^{n+1}(Y_m)(F)\]
and null-homologous means that $[\Delta]$ belongs to $\CH^{n+1}_0(\mathcal{D})$, the kernel 
of the cycle class map $\cl_{\mathcal{D}}^{(n+1)}$ in \eqref{cycle map}.  
 Let $\Ext_{G_F}^1(\cdot,\cdot)$ be the first $\Ext$ functor in the category of continuous $G_F$-representations. 
 For a $G_{F}$-representation $M$, let $M(i)$ denote its $i$-th Tate twist. 
One may associate to $[\Delta]$ the isomorphism class 
in
$$\mathrm{Ext}^1_{G_F}\left(\Q_p,\epsilon_* \cdot H^{2n+1}_\text{\'et}(\overline{Y}_m,\Q_p(n+1))\right)=H^1\left(F,H_p^{2n+1}(\mathcal{D})(n+1)\right)$$
of the extension 
\begin{equation}\label{aj1}
0\longrightarrow \epsilon_*\cdot H^{2n+1}_\text{\'et}\left(\overline{Y}_m,\Q_p(n+1)\right)\longrightarrow E\longrightarrow \Q_p\longrightarrow 0\end{equation}
given by the pull-back of the exact sequence (which comes from the Gysin exact sequence 
\cite[Remark 5.4(b)]{Milne})
{\footnotesize{\begin{equation}\label{aj2}
0\longrightarrow \epsilon_*\cdot H^{2n+1}_\text{\'et}\left(\overline{Y}_m,\Q_p(n+1)\right)\longrightarrow \epsilon_* \cdot H^{2n+1}_\text{\'et}\left(\bar{U},\Q_p(n+1)\right)\longrightarrow \epsilon_*\cdot H^{2n+2}_{\overline{\Delta}}\left(\overline{Y}_m,\Q_p(n+1)\right)\longrightarrow0 \end{equation}}}
(where $U=Y_m-\Delta$, $\overline{U}=U\otimes_F\bar{F}$, 
$\bar{\Delta}=\Delta\otimes_F\bar{F}$) 
via the map $\Q_p\rightarrow\epsilon_* \cdot H^{2n+2}_{\overline{\Delta}}\left(\overline{Y}_m,\Q_p(n+1)\right)$ sending $1$ to the cycle class of $\Delta$; see 
\cite[Remark 9.1]{Jannsen} for the definition of the Abel-Jacobi map, and use Proposition \ref{prop3.1} to obtain the above recipe (see also a similar argument using projectors 
as in \cite[\S 3.3]{BDP}). 
This association defines a map,  
called \emph{$p$-adic \'etale Abel-Jacobi map}
\begin{equation}\label{AJ-global}
\cl_{\mathcal{D},0}^{(n+1)}\colon \CH^{n+1}_0(\mathcal{D})(F)\longrightarrow \Ext^1_{G_F}\left(\Q_p,H^{2n+1}_p(\mathcal{D})(n+1)\right)=H^1\left(F,H_p^{2n+1}(\mathcal{D})(n+1)\right).\end{equation}

\subsection{Semistability} 
We now use $p$-adic Hodge theory to describe the restriction of $\AJ_p$ to $\mathrm{CH}^{n+1}(\mathcal{D})(F_v)$, where $v$ is the place of $F$ above $p$ induced by the inclusion $F\subseteq\overline{\Q}\hookrightarrow\C_p$, which for simplicity we assume to be unramified over $p$; here $F_v$ is the completion of $F$ at $v$, which we also assume to contain $\Q_{p^2}$.
The motive $\mathcal{D}$ is then defined over $F_v$, because the prime $p$, being inert in $K$, splits completely in its Hilbert class field $H$. 
Consider the base change of $Y_m$ to $F_v$ that we also denote by $Y_m$ by a slight abuse of notation, 
and the Abel-Jacobi map
\[\cl_{\mathcal{D},0}^{(n+1)}\colon \CH^{n+1}_0(\mathcal{D})(F_v)\longrightarrow \Ext^1_{G_{F_v}}\left(\Q_p,H^{2n+1}_p(\mathcal{D})(n+1)\right) = H^1\left(F_v,H_p^{2n+1}(\mathcal{D})(n+1)\right)\] obtained by restriction. For a $G_{F_v}=\Gal(\bar{F}_v/F_v)$-representation $V$, let $H^1_\mathrm{st}(G_{F_v},V)$ be the semistable Bloch-Kato Selmer group (\cite[\S 3]{BK}, or \cite[page 361]{IS}). 
By a result of Nekov\'a\v{r} \cite[Theorem 3.6]{Nekovar}
(see also \cite[Lemma 7.1]{IS} and the remarks following it), we know that 
the image of $\AJ_p$ is contained in
$H^1_\mathrm{st}\left(F_v,H_p^{2n+1}(\mathcal{D})(n+1)\right)$. 
We have 
\[H^1_\mathrm{st}\left(F_v,H_p^{2n+1}(\mathcal{D})(n+1)\right)\simeq \Ext^1_{\mathrm{Rep}_\mathrm{st}(G_{F_v})}\left(F_v(n+1),H^{2n+1}_p(\mathcal{D})\right)\]
where $\mathrm{Rep}_\mathrm{st}(G_{F_v})$ denotes the category of semistable $p$-adic representations of $G_{F_v}$, and $\Ext^1_{\mathrm{Rep}_\mathrm{st}(G_{F_v})}(\cdot,\cdot)$ is the first Ext functor in this category. 
The functor $D_{\mathrm{st},F_v}$ gives an isomorphism 
\[\Ext^1_{\mathrm{Rep}_\mathrm{st}(G_{F_v})}\left(F_v(n+1),H_p^{2n+1}(\mathcal{D})\right)
\simeq 
\Ext^1_{\mathrm{MF}_{F_v}^{\phi,N}}\left(F_v[n+1],D_{\mathrm{st},F_v}(H_p^{2n+1}(
\mathcal{D}))\right)
\] where now $\Ext^1_{\mathrm{MF}_{F_v}^{\phi,N}}(\cdot,\cdot)$ denotes the first Ext functor in the category $\mathrm{MF}_{F_v}^{\phi,N}$ (\cite[(44)]{IS}), and for an object $M$ in this category, $M[i]$ is its $i$-th fold twist described in \cite[\S2]{IS}. 
By \cite[Lemma 2.1]{IS}, 
\[\Ext^1_{\mathrm{MF}_{F_v}^{\phi,N}} \left(F_v[n+1],D_{\mathrm{st},F_v}(H_p^{2n+1}(
\mathcal{D}))\right)\simeq 
\frac{D_{\mathrm{st},F_v}(H_p^{2n+1}(
\mathcal{D}))}{F^{n+1}\left(D_{\mathrm{st},F_v}(H_p^{2n+1}(
\mathcal{D}))\right)}.\] Therefore we conclude that 
\[H^1_\mathrm{st}\left(F_v,H_p^{2n+1}(\mathcal{D})(n+1)\right)\simeq 
\frac{D_{\mathrm{st},F_v}(H_p^{2n+1}(
\mathcal{D}))}{F^{n+1}\left(D_{\mathrm{st},F_v}(H_p^{2n+1}(
\mathcal{D}))\right)}.\] 
Finally, using the canonical map 
$D_{\mathrm{st},F_v}
\left(H_p^{2n+1}(\mathcal{D})\right)\hookrightarrow D_{\mathrm{st},\hat{\Q}_p^\unr}
\left(H_p^{2n+1}(\mathcal{D})\right)$ (which respects the filtrations on both sides) and 
\eqref{F n+1}, we obtain from $\cl_{\mathcal{D},0}^{(n+1)}$ a map $\AJ_p$ still called 
$p$-adic Abel-Jacobi map, 
\begin{equation}\label{Abel-Jacobi} 
\AJ_p\colon \CH^{n+1}(\mathcal{D})(F_v)\longrightarrow
\left(M_k(\Gamma)\otimes (\mathcal{V}_n)_{z_{A_0}}\right)^\vee.\end{equation}

\subsection{The de Rham realization} 
We now introduce, following 
\cite{IS}, a more concrete description of the map \eqref{Abel-Jacobi}. 
Fix a point $x_A\in X_M(F)$ (as above, $F\subseteq\bar{\Q}$) 
which  reduces to a non-singular point in the special fiber of $X_M$, and let $A^m\times A_0^m$ 
be the fiber of $\mathcal{A}^m\times A_0^m\rightarrow X_M$ at $x_A$.  
Define \[H^1(\overline{X}_M,\mathbb{L}_{n,n})= H^1(\overline{X}_M,\mathbb{L}_{n})\otimes
(\mathbb{L}_n)_{\bar{x}_{A_0}}.\] 
Let $\bar{x}_A=x_A\otimes_F\bar{F}$, $U_{x_A}=X_M-\{x_A\}$ and $\overline{U}_{x_A}=U_{x_A}\otimes_{F}\bar{F}$. 
The Gysin sequence gives rise to an exact sequence 
\[0\longrightarrow H^1(\overline{X}_M,\mathbb{L}_{n,n})(n+1)\longrightarrow 
H^1(\overline{U}_{x_A},\mathbb{L}_{n,n})(n+1)\longrightarrow \left((\mathbb{L}_n)_{\bar{x}_A}\otimes (\mathbb{L}_n)_{\bar{x}_{A_0}}\right)(n)\longrightarrow 0\]
whose surjectivity follows from the analogous exact sequence in 
\cite[(51)]{IS} tensoring with the constant sheaf $(\mathbb{L}_n)_{\bar{x}_{A_0}}$. Applying the projector $(p_G)_*$ we obtain an exact sequence 
\begin{equation}\label{37IS}
0\longrightarrow H_p(\mathcal{D})(n+1)\longrightarrow E\longrightarrow 
\left((\mathbb{L}_n)_{\bar{x}_A}\otimes (\mathbb{L}_n)_{\bar{x}_{A_0}}\right)(n)\longrightarrow 0.\end{equation} 

Suppose $F\subseteq \hat{\Q}_p^\unr$. Let $z_A$ and $z_{A_0}$ be the points in $\mathcal{H}_p(\hat{\Q}_p^\unr)$ lying over $x_A$ and $x_{A_0}$, respectively (using \eqref{split}). 
Define  
$U_{z_A}=X_\Gamma-\{z_A\}$ and put 
\begin{equation}\label{iso2}
H^1_\mathrm{dR}(U_{z_A},\mathcal{V}_{n,n})=
H^1_\mathrm{dR}(U_{z_A},\mathcal{V}_{n})\otimes (\mathcal{V}_n)_{z_{A_0}}.\end{equation}
Let $\mathrm{Res}_{z}:H^1_\mathrm{dR}(U,\mathcal{V}_n)\rightarrow (\mathcal{V}_n)_{z}$ be the residue map at a point $z\in X_\Gamma(\hat{\Q}_p^\unr)$.  
The Gysin sequence of \cite[Theorem 5.13]{IS} gives rise, after tensoring with $(\mathcal{V}_n)_{z_{A_0}}$
and using \eqref{iso1}, \eqref{iso2}, to an exact sequence 
in $\mathrm{MF}^{\phi,N}_{\hat{\Q}_p^\unr}$: 
{\footnotesize{\begin{equation}\label{5.13IS}
0\longrightarrow H^1_{\mathrm{dR}}(X_\Gamma,\mathcal{V}_{n,n})[-(n+1)]\longrightarrow
H^1_\mathrm{dR}(U_{z_A},\mathcal{V}_{n,n})[-(n+1)]\overset{\mathrm{Res}_{z_A}}\longrightarrow
\left((\mathcal{V}_n)_{z_A}\otimes (\mathcal{V}_{n})_{z_{A_0}}\right)[-n]\longrightarrow 0.\end{equation}}}
This exact sequence is obtained by applying $D_{\mathrm{st},\hat{\Q}_p^\unr}$ to \eqref{37IS}. 

\begin{remark}
The shift in \eqref{5.13IS} is due to the definition of Tate twists adopted in \cite[page 337]{IS}; see \cite[\S7.1.3]{FO} or \cite[\S8.3]{BrCo} for a different convention.  
\end{remark} 

We have the cycle class map 
{\footnotesize{\[\begin{split}
\cl=\cl_{(A^m\times A_0^m,\epsilon_M)}^{(n)}: \CH^{n}((A^m\times A_0^m,\epsilon_M))\longrightarrow &
\Gamma\left(D_{\mathrm{st},\hat{\Q}_p^\unr}(H^{2n}_p((A^m\times A_0^m,\epsilon_M)(n)))\right)\\
& \simeq\Gamma\left(D_{\mathrm{st},\hat{\Q}_p^\unr}(H^{2n}(\overline{X}_M,(\mathbb{L}_n)_{x_A}\otimes (\mathbb{L}_n)_{x_{A_0}})(n))\right)\\
& \simeq\Gamma\left(D_{\mathrm{st},\hat{\Q}_p^\unr}(((\mathbb{L}_n)_{\bar{x}_A}\otimes (\mathbb{L}_n)_{\bar{x}_{A_0}})(n)))\right)\\
& \simeq\Gamma\left(((\mathcal{V}_n)_{z_A}\otimes(\mathcal{V}_n)_{z_{A_0}})[-n]\right).
\end{split}\]}}
Next, from \eqref{5.13IS} we obtain a connecting homomorphism in the sequence of $\Ext$ groups  
\[\begin{split}
\Gamma\left(((\mathcal{V}_n)_{z_A}\otimes(\mathcal{V}_n)_{z_{A_0}})[-n]\right)\overset\partial\longrightarrow &
\Ext^1_{\mathrm{MF}_{\hat{\Q}_p^\unr}^{\phi,N}}\left(\hat{\Q}_p^\unr,H^1_{\mathrm{dR}}(X_\Gamma,\mathcal{V}_{n,n})[-(n+1)]\right)\\
&\simeq \Ext^1_{\mathrm{MF}_{\hat{\Q}_p^\unr}^{\phi,N}}\left(\hat{\Q}_p^\unr[n+1],H^1_{\mathrm{dR}}(X_\Gamma,\mathcal{V}_{n,n})\right)\\
&\simeq \left(M_k(\Gamma)\otimes(\mathcal{V}_n)_{z_{A_0}}\right)^\vee\end{split}\]
where the last isomorphism comes, as before, from \eqref{F n+1} and \cite[Lemma 2.1]{IS}.   
On the other hand, we have a canonical map 
\[
i:\CH^{n}((A^m\times A_0^m,\epsilon_M))\longrightarrow \CH^{n+1}(\mathcal{D}) .\]
The definition of the Abel-Jacobi map (\cite[\S9]{Jannsen}) shows 
that the following diagram is commutative: 
\begin{equation}\label{IS7.3}\xymatrix{
i^{-1}\left(\CH^{n+1}_0(\mathcal{D})\right)\ar[r]^-{\cl}\ar[d]^i  & 
\Gamma\left(((\mathcal{V}_n)_{z_A}\otimes(\mathcal{V}_n)_{z_{A_0}})[-n]\right)\ar[d]^\partial\\
\CH^{n+1}_0(\mathcal{D})\ar[r]^-{\AJ_p}&\left(M_k(\Gamma)\otimes(\mathcal{V}_n)_{z_{A_0}}\right)^\vee.}
\end{equation}
Suppose that $\Delta$ is supported in the fiber of $\mathcal{D}$ above $x_A\in X_M(F)$, then $\AJ_p(\Delta)$ is the 
extension class determined by the following diagram (in which the right square is cartesian) 
\begin{equation}\label{diagramAJ}
\xymatrix{0\ar[r] & 
H^1_{\mathrm{dR}}(X_\Gamma,\mathcal{V}_{n,n})\ar[r]^{j_*} & 
H^1_\mathrm{dR}(U_{z_A},\mathcal{V}_{n,n})\ar[r]^-{{\mathrm{Res}_{z_A}}}& 
\left((\mathcal{V}_n)_{z_A}\otimes(\mathcal{V}_n)_{z_{A_0}}\right)[1]\ar[r]& 0\\
0\ar[r] & 
H^1_{\mathrm{dR}}(X_\Gamma,\mathcal{V}_{n,n})\ar[r] \ar@{=}[u]& 
E\ar[u]\ar[r]&\hat{\Q}_p^\unr[n+1]\ar[u] \ar[r]&0}\end{equation}
where the vertical left map sends $1\longmapsto\cl(\Delta)[n+1]$.

\section{Generalized Heegner cycles}\label{section: generalized Heegner cycles}

\subsection{Definition of the cycles} We fix a field $F$ containing 
the Hilbert class field $H$ of $K$. Recall the fixed abelian surface $A_0$ with QM and complex multiplication by $\mathcal{O}_K$. 
Consider the set of pairs $(\varphi,A)$, where $A$ is an abelian surface with QM and $\varphi\colon A_0\rightarrow A$ is a false isogeny (defined over $\bar{K}$) of false elliptic curves, of degree prime to $N^+M$, \emph{i.e.} whose kernel intersects the level structures of $A_0$ trivially. 
Let $x_{A}$ be the point on $X_M$ corresponding to $A$ with level structure given by composing $\varphi$ with the level structure of $A_0$. We associate to any pair $(\varphi,A)$ a codimension $n+1$ cycle $\Upsilon_\varphi$ on $Y_m$ by defining
$$\Upsilon_\varphi:=(\Gamma_\varphi)^m\subset (A_0\times A)^m\simeq A^m\times A^m_0\subset\mathcal{A}^m\times A^m_0,$$
where $\Gamma_\varphi\subset A_0\times A$ is the graph of $\varphi$ and the inclusion $A^m\times A_0^m\subset\mathcal{A}^m\times A_0^m$ is $\mathrm{id}_{A_0}^m$ on the second component. We then set
$$\Delta_\varphi:=\epsilon\Upsilon_\varphi.$$
The cycle $\Delta_\varphi$ of $\mathcal{D}$ is supported on the fiber above $x_{A}$ and has codimension $n+1$ in $\mathcal{A}^m\times A_0^m$, thus $\Delta_\varphi\in \mathrm{CH}^{n+1}(\mathcal{D})$. Since the cycle class map sends $\Delta_\varphi$ to the $p$-adic realization $H_p^{2n+2}(\mathcal{D})$ and $H_p^{2n+2}(\mathcal{D})=0$, the cycle $\Delta_\varphi$ is homologous to zero. 

\subsection{The image of $\Delta_\varphi$ under the $p$-adic Abel-Jacobi map}\label{sec5.2}

For any $D\in\mathrm{MF}_{\hat{\Q}_p^\unr}^{\phi,N}$, write 
$D=\oplus_{\lambda}D_\lambda$ for its slope decomposition, 
where $\lambda\in\Q$
(\cite[(2)]{IS}). 
Recall the monodromy operator $N$ introduced in \eqref{monodromy}. 
\begin{lemma}\label{lemma slope}  $N$ induces an isomorphism 
$H^1_\mathrm{dR}(X_\Gamma,\mathcal{V}_{n,n})_{n+1}\simeq 
H^1_\mathrm{dR}(X_\Gamma,\mathcal{V}_{n,n})_{n}$. 
\end{lemma}
\begin{proof}
Since the monodromy operator $N$ and the Frobenius $\Phi$ on $H^1_\mathrm{dR}(X_\Gamma,\mathcal{V}_{n,n})$ satisfy the relation $N\Phi=p\Phi N$, we have $N\left(H^1_\mathrm{dR}(X_\Gamma,\mathcal{V}_{n,n})_{n+1}\right)\subseteq
H^1_\mathrm{dR}(X_\Gamma,\mathcal{V}_{n,n})_{n}$. Since 
$(\mathcal{V}_n)_{z_{A_0}}$ is isotypical of slope $n/2$, we have 
$$H^1_\mathrm{dR}(X_\Gamma,\mathcal{V}_{n,n})_{n+1}=H^1_\mathrm{dR}(X_\Gamma,\mathcal{V}_n)_{\frac{n}{2}+1}\otimes (\mathcal{V}_n)_{z_{A_0}}$$
and 
$$H^1_\mathrm{dR}(X_\Gamma,\mathcal{V}_{n,n})_{n}=H^1_\mathrm{dR}(X_\Gamma,\mathcal{V}_n)_{\frac{n}{2}}\otimes (\mathcal{V}_n)_{z_{A_0}}.$$  
By \cite{IS}, we know that $N_n\colon H^1_\mathrm{dR}(X_\Gamma,\mathcal{V}_n)_{\frac{n}{2}+1}\rightarrow H^1_\mathrm{dR}(X_\Gamma,\mathcal{V}_n)_\frac{n}{2}$ is an isomorphism, thus the restriction of $N$ to $H^1_\mathrm{dR}(X_\Gamma,\mathcal{V}_{n,n})_{n+1}$ is an isomorphism by the definition of the monodromy operator $N$ given in \eqref{monodromy}. 
\end{proof}

Fix $f\in M_k(\Gamma)$ and $v\in(\mathcal{V}_n)_{z_{A_0}}$.
Thanks to Lemma \ref{lemma slope}, we can apply 
\cite[Lemma 2.1]{IS} (see also \cite[Lemma 3.3]{Masdeu}) to compute 
$\AJ_p(\Delta_\varphi)( f\otimes v)$. 
With the notation as in \eqref{diagramAJ}, and following \emph{loc. cit}, choose $\alpha\in H^1_\mathrm{dR}(U_{z_A},\mathcal{V}_{n,n})_{n+1}$ such that 
\[\mathrm{Res}_{z_A}(\alpha)=\cl_A(\Delta_\varphi)\] and $N(\alpha)=0$.  
Choose $\beta$ in $H^1_\mathrm{dR}(X_\Gamma,\mathcal{V}_{n,n})$ such that 
\[j_*(\beta)\equiv\alpha\mod F^{n+1}\left(H^1_\mathrm{dR}(U_{z_A},\mathcal{V}_{n,n})\right).\]
Then the image of the extension $\cl_{\mathcal{D},0}^{(n+1)}(\Delta_\varphi)$ in 
\[H^1_\mathrm{dR}(X_\Gamma,\mathcal{V}_{n,n})/F^{n+1}\left(H^1_\mathrm{dR}(X_\Gamma,\mathcal{V}_{n,n})\right)\simeq (M_k(\Gamma)\otimes(\mathcal{V}_n)_{z_{A_0}})^\vee\] is the class of $\beta$ (which we denote by the same symbol $\beta$) in this quotient.  
Let $\omega_f$ be the class in 
$F^{n+1}\left(H^1_\mathrm{dR}(X_\Gamma,\mathcal{V}_n)\right)$ corresponding to 
$f\in M_k(\Gamma)$ under the isomorphism \eqref{de Rham filtration}. 
Recall the pairing $\langle,\rangle_{\mathcal{V}_{n,n}}$ defined in \eqref{pairing V_n}. 
Then by definition 
\begin{equation}\label{first equality AJ}
\AJ_p(\Delta_\varphi)(f\otimes v)=\langle\omega_f\otimes v,\beta\rangle_{\mathcal{V}_{n,n}}.\end{equation}

From the proof of \cite[Theorem 6.4]{IS} we know that 
$H^1_\mathrm{dR}(X_\Gamma,\mathcal{V}_n)$ 
decomposes as the direct sum of $H_\mathrm{dR}^1(X_\Gamma,\mathcal{V}_n)_{\frac{n}{2}}$ and $F^{\frac{n}{2}+1}\left(H_\mathrm{dR}^1(X_\Gamma,\mathcal{V}_n)\right)$. 
Since $$F^{\frac{n}{2}+1}\left(H_\mathrm{dR}^1(X_\Gamma,\mathcal{V}_n)\right)=F^{n+1}\left(H_\mathrm{dR}^1(X_\Gamma,\mathcal{V}_n)\right)$$ and $F^{n+1}((\mathcal{V}_n)_{z_{A_0}})=0$, using the previous decomposition, and the fact that, as above,  
$(\mathcal{V}_n)_{z_{A_0}}$ is isotypical of slope $n/2$, we obtain a decomposition 
\[H^1_\mathrm{dR}(X_\Gamma,\mathcal{V}_{n,n})\simeq H^1_\mathrm{dR}(X_\Gamma,\mathcal{V}_{n,n})_n\oplus F^{n+1}\left(H^1_{\mathrm{dR}}(X_\Gamma,\mathcal{V}_{n,n})\right).\]
We may therefore assume that the element $\beta$ considered above belongs to $H^1_\mathrm{dR}(X_\Gamma,\mathcal{V}_{n,n})_n$. 
Moreover, again from the proof of \cite[Theorem 6.4]{IS}
we know that 
\begin{equation}\label{eq25}
\ker(N_n)=\iota\left(H^1(\Gamma,(V_n)_{\Q_p^\mathrm{unr}})\right)=H^1_\mathrm{dR}(X_\Gamma,\mathcal{V}_n)_\frac{n}{2}\end{equation}
where $\iota$ is the map considered in \eqref{iota}. 
To simplify the notation we put 
\[H^1(\Gamma,V_{n,n})=H^1\left(\Gamma,(V_n)_{\hat{\Q}_p^\mathrm{unr}}\right)\otimes(\mathcal{V}_n)_{z_{A_0}}.\]
We now extend $\iota$ to a map, still denoted by the same symbol, 
\[\iota=\iota\otimes\mathrm{id}_{(\mathcal{V}_n)_{z_{A_0}}}\colon H^1(\Gamma,V_{n,n})\longmono H^1_\mathrm{dR}(X_\Gamma,\mathcal{V}_{n,n})\]
and \eqref{eq25} shows that there exists an isomorphisms 
 \[\ker(N)=\iota\left(H^1(\Gamma,V_{n,n})\right)=H^1_\mathrm{dR}(X_\Gamma,\mathcal{V}_{n,n})_n.\]
 Therefore we may assume $\beta=\iota(c)$ for some $c\in  H^1(\Gamma,V_{n,n})$. 
 
We now introduce still an other pairing $\langle,\rangle_\Gamma$. Let $C_\mathrm{har}(V_n)^\Gamma$ denote the 
$\Q_p$-vector space of $\Gamma$-invariant $V_n$-valued 
harmonic cocycles (see for example \cite[Definition 2.2.9]{DasgTeit}). 
We denote 
\[\langle,\rangle_{\Gamma}': 
C_\mathrm{har}(V_{n})^\Gamma\otimes H^1(\Gamma,V_n)\longrightarrow \Q_p\]
the pairing introduced in \cite[(75)]{IS}. 
To simplify the notation, we set 
\[C_\mathrm{har}(V_{n,n})^\Gamma=C_\mathrm{har}(V_{n})^\Gamma\otimes(\mathcal{V}_n)_{z_{A_0}}.\]
We then define the pairing 
\[\langle,\rangle_{\Gamma}: 
C_\mathrm{har}(V_{n,n})^\Gamma\otimes H^1(\Gamma,V_{n,n})\longrightarrow \Q_p\]
by $\langle,\rangle_{\Gamma}=\langle,\rangle_{\Gamma}'\otimes\langle,\rangle_{V_n}$ (where as above we identify $(\mathcal{V}_n)_{z_{A_0}}$ and $V_n$). 
Recall the map $I$ is defined in \eqref{I}. 

\begin{lemma} \label{lemma5.2}
$\langle\omega_f\otimes v,\beta\rangle_{\mathcal{V}_{n,n}}=-\langle I(\omega_f)\otimes v,c\rangle_\Gamma.$\end{lemma}

\begin{proof}
Write $\beta=\sum_i\beta_i\otimes v_i$ and $c=\sum_jc_j\otimes w_j$. 
The assumption $\iota(c)=\beta$ shows that $i=j$, $v_i=w_i$ and 
$\iota(c_i)=\beta_i$ where here $\iota$ is the map in \eqref{iota}. 
By \cite[Theorem 10.2]{IS} we know that for each $i$ we have 
$\langle \omega_f,\beta_i\rangle_{\mathcal{V}_n}=-\langle I(\omega_f),c_i\rangle'_\Gamma$. 
The definitions of $\langle,\rangle_{\mathcal{V}_{n,n}}$ and $\langle,\rangle_\Gamma$ imply the result. 
\end{proof}

Recall the open set $U_{z_A}=X_\Gamma-\{z_A\}$. 
Write $\alpha-j_*(\beta)=\sum_i\gamma_i\otimes v_i$. 
For each $i$, let $\chi_i$ be a $\Gamma$-invariant 
${V}_{n}$-valued meromorphic differential 
form on $\mathcal{H}_p$ which is holomorphic outside 
$\pi^{-1}(U_{z_A})$, with a simple pole at $z_A$, 
and whose class $[\chi_i]$ in $F^{\frac{n}{2}+1}\left(H^1_\mathrm{dR}(U_{z_A},\mathcal{V}_{n})\right)$
represents $\gamma_i$. Then the class of $\chi=\sum_i\chi_i\otimes v_i$ represents $\alpha-j_*(\beta)$. 

Having identified 
$H^1_\mathrm{dR}(X_\Gamma,\mathcal{V}_n)$ with the $\hat{\Q}_p^\unr$-vector space of $\Gamma$-invariant $V_n$-valued differential forms of the second kind on $\mathcal{H}_p$ modulo exact forms, denote $F_{\omega_f}\in H^0(X_\Gamma,\mathcal{V}_n)$ the Coleman primitive of $\omega_f$ (\cite[\S 2.3]{deshalit}).  
Having fixed $n$, we write $\langle,\rangle_{z,z_{A_0}}$ for the restriction of 
$\langle,\rangle_{\mathcal{V}_{n,n}}$ to the stalk of $\mathcal{V}_{n,n}$ at $z$. 
Then $\langle,\rangle_{z,z_{A_0}}$ is a pairing on $(\mathcal{V}_n)_{z}\otimes(\mathcal{V}_n)_{z_{A_0}}$. 

\begin{lemma}\label{lemma5.3} 
$-\langle I(\omega_f)\otimes v,c\rangle_\Gamma=\langle F_{\omega_f}(z_A)\otimes v,\mathrm{Res}_{z_A}(\chi)\rangle_{{z_A,z_{A_0}}}$. \end{lemma}

\begin{proof}
As in the proof of Lemma \ref{lemma5.2} write $c=\sum_jc_j\otimes w_j$. By definition, 
\[\langle I(\omega_f)\otimes v,c\rangle_\Gamma=\sum_j\langle I(\omega_f),c_j\rangle_\Gamma'\cdot\langle v,w_j\rangle_{V_n}.\]
By \cite[Corollary 10.7]{IS},  
\[\langle I(\omega_f),c_j\rangle_\Gamma'=\langle F_{\omega_f}(z_A),\mathrm{Res}_{z_A}(\chi_j)\rangle_{{V}_n}\]
where in the last pairing we identify $(\mathcal{V}_n)_{z_{A}}$ with $V_n$. The result 
follows now from the definition of the pairing $\langle,\rangle_{\mathcal{V}_{n,n}}$ in 
\eqref{pairing V_n}.
\end{proof}

For a smooth projective variety $X$ defined over $F$, denote 
\[\cup:H^p_\mathrm{dR}(X)\otimes H^q_\mathrm{dR}(X)\longrightarrow H^{p+q}_\mathrm{dR}(X)\] the cup product pairing on the de Rham cohomology of $X$. 
If $d$ is the dimension of $X$, we also denote $\eta_X:H^{2d}(X)\rightarrow F$ the 
trace isomorphism. 

Let $A_z$ be the fiber at $z$ of $\mathcal{A}\rightarrow X_M$.
The projector $\epsilon$ defines a projector $\epsilon_z$ on 
$A_z^m$ and we have (\cite[Theorem 5.8 (iii)]{Besser})
\begin{equation}\label{eq28}(\epsilon_z)_* H^n_{\mathrm{dR}}(A_z^m)\simeq (\mathcal{V}_n)_z.\end{equation}
We also have a canonical map 
\begin{equation}\label{eq29}
(\epsilon_{z_A})_* H^n_{\mathrm{dR}}(A^m)\otimes 
(\epsilon_{z_{A_0}})_* H^n_{\mathrm{dR}}(A^m_0)\longmono
H^n_{\mathrm{dR}}(A^m)\otimes H^n_{\mathrm{dR}}(A^m_0)\longmono
H^{2n}_{\mathrm{dR}}(A^m\times A_0^m)\end{equation} arising from the Kunneth decomposition; 
explicitly, this is the map which takes $\alpha\otimes \beta$ to $p_A^*(\alpha)\cup p_{A_0}^*(\beta)$, where $p_A:A^m\times A_0^m\rightarrow A^m$ 
and $p_{A_0}:A^m\times A_0^m\rightarrow A_0^m$ are the two projections. 
Composing \eqref{eq28} with \eqref{eq29} we obtain a map 
\[\Theta:(\mathcal{V}_n)_{z_A}\otimes(\mathcal{V}_n)_{z_{A_0}}\longmono 
H^{2n}_{\mathrm{dR}}(A^m\times A_0^m).\]

Recall that, given a false isogeny $\varphi:A_0\rightarrow A$, we have pull-back an push-forward maps 
$\varphi^*:H^i_\mathrm{dR}(A)\rightarrow H^i_\mathrm{dR}(A_0)$ and 
$\varphi_*:H^i_\mathrm{dR}(A_0)\rightarrow H^i_\mathrm{dR}(A)$. Applying the projectors 
$\epsilon_{z_{A_0}}$ and $\epsilon_{z_A}$ and using \eqref{eq28} 
we thus obtain maps 
$\varphi^*:(\mathcal{V}_n)_{z_A}\rightarrow (\mathcal{V}_n)_{z_{A_0}}$ and 
$\varphi_*:(\mathcal{V}_n)_{z_{A_0}}\rightarrow (\mathcal{V}_n)_{z_{A}}$. 

\begin{lemma}\label{lemma5.4}
Fix $v_A\otimes v_{A_0}\in (\mathcal{V}_n)_{z_A}\otimes(\mathcal{V}_n)_{z_{A_0}}$ and an isogeny $\varphi:A_0\rightarrow A$. Then 
\[\langle v_A\otimes v_{A_0},\cl(\Delta_\varphi)\rangle_{z_A,z_{A_0}}=
\langle v_A,\varphi_*(v_{A_0})\rangle_{z_{A}}.\]
\end{lemma}

\begin{proof}
For each $z$, the pairing $\langle,\rangle_{z}$ 
on $(\mathcal{V}_n)_{z}$ is induced
by the pairing on $V_n$ and the isomorphism $(\mathcal{V}_n)_{z}\simeq V_n$ corresponds under the above map 
to the cup product pairing $\cup$ on the de Rham cohomology of $A_z^m$ (see \cite[Remark 5.12]{IS}). 
Let 
$\varrho=(\varphi^m,\mathrm{id}^m):A_0^m\rightarrow A^m\times A_0^m$.   Then 
we have $\varrho(A_0^m)=\Upsilon_\varphi$ and 
$\varrho_*(1_{A_0^m})=\cl_{A^m\times A_0^m}(\Upsilon_\varphi)$, 
where $1_{A_0^m}\in H^0_\mathrm{dR}(A_0^m)$ is the identity element. 
Thus
\[\begin{split}
\langle v_A\otimes v_{A_0},\cl(\Delta_\varphi)\rangle_{z_A,z_{A_0}}&=
\eta_{A^m\times A_0^m}\left(
\Theta(v_A\otimes v_{A_0})\cup \left(\cl_{A^m\times A_0^m}(\Upsilon_\varphi)\right)\right)\\
&=\eta_{A^m\times A_0^m}\left(\Theta(v_A\otimes v_{A_0})\cup\varrho_*(1_{A_0^m})\right)\\
&=
\eta_{A^m\times A_0^m}\left(\left(p_A^*(v_A)\cup p_{A_0}^*(v_{A_0})\right)\cup\varrho_*(1_{A_0^m})\right).
\end{split}\] 
It turns out that 
\[\begin{split}
\eta_{(A\times A_0)^m}\left(\left(p_A^*(v_A)\cup p_{A_0}^*(v_{A_0})\right)\cup\varrho_*(1_{A_0^m})\right)&= \eta_{A_0^m}(\varrho^*(p_A^*(v_A)\cup p_{A_0}^*(v_{A_0}))\cup 1_{A_0^m})\\
&=\eta_{A_0^m}(\varphi^*(v_A)\cup v_{A_0})
\end{split}
\]
Therefore  
\[
\langle v_A\otimes v_{A_0},\cl(\Delta_\varphi)\rangle_{z_A,z_{A_0}}=
\eta_{A_0^m}\left(\varphi^*(v_A)\cup v_{A_0}\right)=
\eta_{A^m}\left(v_A\cup\varphi_*(v_{A_0})\right).
\]
Now the term on the right of the last displayed equation coincides with $\langle v_A,\varphi_*(v_{A_0})\rangle_{z_{A}}$, and the result follows.
\end{proof}
%

\begin{theorem}\label{Theorem1}
Let $\varphi:A_0\rightarrow A$ and $v\in(\mathcal{V}_n)_{z_{A_0}}$.  
Then 
\[\AJ_p(\Delta_\varphi)(f\otimes v)=\langle F_{\omega_f}(z_A),\varphi_*(v)\rangle_{z_{A}}.\]
\end{theorem}
\begin{proof} 
Recall that $\mathrm{Res}_{z_A}(\chi)=\mathrm{Res}_{z_A}(\alpha)=\cl(\Delta_\varphi)$, where 
the first equality follows because $\mathrm{Res}_{z_A}\left(j_*(\beta)\right)=0$. 
Combining this with \eqref{first equality AJ}, Lemma \ref{lemma5.2} and Lemma \ref{lemma5.3} we obtain \[\AJ_p(\Delta_\varphi)(f\otimes v)=\langle F_{\omega_f}(z_A)\otimes v,\cl(\Delta_\varphi)\rangle_{z_A,z_{A_0}}.\]
The result follows then from Lemma \ref{lemma5.4}.\end{proof}

\begin{corollary}\label{coro5.6}
Let $\varphi:A_0\rightarrow A$, $\varphi^\vee:A\rightarrow A_0$ the 
dual isogeny, and $v\in(\mathcal{V}_n)_{z_{A}}$.  
Denote $\deg(\varphi)$ the degree of $\varphi$. 
Then 
\[\AJ_p(\Delta_\varphi)(f\otimes \varphi^\vee_*(v))=\deg(\varphi)\cdot\langle F_{\omega_f}(z_A),v\rangle_{z_{A}}.\]
\end{corollary}

\begin{proof} Let $\deg(\varphi)$ denote multiplication 
by $\deg(\varphi)$ map on $A$ and $A_0$. 
The result follows from Theorem \ref{Theorem1} observing that 
$\deg(\varphi)_*=(\varphi\circ\varphi^\vee)_*=\varphi_*\circ\varphi_*^\vee$. 
\end{proof}

\section{Anticyclotomic $p$-adic $L$-functions}
\label{section 6}
This section contains the main result of this paper, 
in which we connect our generalised Heegner cycles 
to certain semidefinite integrals and anticyclotomic 
$p$-adic $L$-functions extensively studied in the literature, especially in \cite{BD-Heegner}, \cite{BD-CD}, \cite{BD}, \cite{BDIS}, \cite{IS}, \cite{Sev}. The setting is as before: $N=pN^+N^-$ is a factorisation of the integer $N\geq 1$ into coprime integers $p,N^+,N^-$ with $p\nmid N^+N^-$ a prime number, and $N^-$ be a square-free product of an odd number of factors; $K$ is a quadratic imaginary field such that all primes dividing $N^+$ (respectively, $pN^-$) are split (respectively, inert) in $K$. 
We also fix an integer $k_0\geq 4$, and a modular form $f$ of level $\Gamma_0(N)$ and weight $k_0$. We put $n_0=k_0-2$ and $m_0=n_0/2$. 

\subsection{Measure valued modular forms} 
We begin by recalling some results from \cite{BD} and \cite{Sev}, to which the reader is referred to for details. Let $\mathcal{D}(\Z_p^\times)$ the $\Q_p$-algebra of locally analytic distributions on $\Z_p^\times$. 
For each $\Z_p$-lattice $L\subseteq\Q_{p}^2$, denote $L'$ the subset of $L$ consisting of primitive vectors (if $L=\Z_p v_1\oplus \Z_p v_2$, then $L'$ consists of those 
$v=av_1+bv_2$ such that at least one of $a$ and $b$ is not divisible by $p$). 
For each lattice $L$, denote by $\mathcal{D}(L')$ the $\Q_p$-vector space of locally analytic distributions on $L'$, \emph{i.e.} $\mathcal{D}(L')=\Hom_{\Q_p\text{-}\mathrm{cont}}(\mathcal{A}(L'),\Q_p)$, where $\mathcal{A}(L')$ is the $\Q_p$-vector space of $\Q_p$-valued locally analytic functions on $L'$. 
Since $L'$ is $\Z_p^\times$-stable, there is a natural $\mathcal{D}(\Z_p^\times)$-module structure on $\mathcal{D}(L')$, defined by the formula
$$\int_{L'} F(x,y)d(r\mu)(x,y):=\int_{\Z_p^\times}\left(\int_{L'}F(tx,ty)d\mu(x,y)\right)dr(t).$$
Let $A(U)$ be the $\Q_p$-affinoid algebra of an open affinoid disk $U\subset\mathcal{W}$, where 
\[\mathcal{W}:=\Hom_\mathrm{cont}(\Z_p^\times,\Q_p^\times).\]  
We view $\Z\subseteq\mathcal{W}$ via the map which takes $k$ to the homomorphism 
$x\mapsto x^{k-2}$. 
The $\Q_p$-affinoid algebra $A(U)$ has a $\mathcal{D}(\Z_p^\times)$-module structure given by the map 
$\mathcal{D}(\Z_p^\times)\rightarrow A(U)$ defined by 
$r\mapsto \left[\kappa\mapsto\int_{\Z_p^\times}\kappa(t)dr(t)\right].$
Let \[\mathcal{D}(L',U):=A(U)\hat{\otimes}_{\mathcal{D}(\Z_p^\times)}\mathcal{D}(L').\] 

Let $B$ be the definite quaternion algebra over $\Q$ with discriminant $N^-$, and let $R$ be a fixed Eichler $\Z[1/p]$-order of level $N^+$ in $B$. Fix an Eichler $\Z$-order $\underline{R}$ of $B$ of level $N^+$ in such a way that $\underline{R}[1/p]=R$, and let $\cO_B$ be a maximal $\Z$-order of $B$ containing $\underline{R}$. We will write $\hat{\underline{R}}$ for the adelisation $\underline{R}\otimes\hat{\Z}$ of $\underline{R}$. For each prime number $\ell\nmid N^-$ fix a $\Q_\ell$-algebra isomorphisms $\iota_\ell\colon B\otimes\Q_\ell\overset\sim\rightarrow\M_2(\Q_l)$ sending $\cO_B\otimes \Z_\ell$ isomorphically onto $\M_2(\Z_\ell)$. Write $\hat{\Q}$ for the ring of finite ad\'eles of $\Q$ and $\hat{B}$ for $B\otimes\hat{\Q}$. Define the level structures $\Sigma=\Sigma(N^+p,N^-)=\prod_\ell\Sigma_\ell$ for 
$$\Sigma_\ell=
\begin{cases}
(\cO_B\otimes\Z_\ell)^\times & \text{if}\ \ell\nmid N^+p\\
\iota_\ell^{-1}(\Gamma_0(N^+p\Z_\ell)) & \text{if}\ \ell\mid N^+p
\end{cases}$$
where $\Gamma_0(N^+p\Z_\ell)$ denotes the subgroup of $\GL_2(\Z_\ell)$ 
consisting of matrices which are upper triangular modulo $N^+p$. Write $\Sigma_\infty$ to denote the open compact subroup obtained from the group $\Sigma$ by replacing the local condition at $p$ with the local condition $\iota_p(\Sigma_{\infty,p})=\GL_2(\Z_p)$. 
Let $S$ be any commutative ring, and $A$ be any $S$-module with an $S$-linear left action of the semigroup $\M_2(\Z_p)$ of matrices with entries in $\Z_p$ and non-zero determinant. We define the $S$-module $S(\Sigma,A)$ as the space of $A$-valued automorphic forms on $B^\times$ of level $\Sigma$, \emph{i.e.}
$$S(\Sigma,A)=\{\phi\colon \hat{B}^\times\rightarrow A\ :\ \phi(gb\sigma)=\iota_p(\sigma_p^{-1})\phi(b)\},$$
where $g\in B^\times$ (embedded diagonally in $\hat{B}^\times$), $b\in \hat{B}^\times$ and $\sigma\in \Sigma$. Observe that, by the strong approximation theorem for $B$, $\hat{B}^\times=B^\times B_p^\times\Sigma$ and a modular form $\phi$ in $S(\Sigma,A)$ can be viewed as a function on $R^\times\backslash B_p^\times/\iota_p^{-1}(\Gamma_0(p\Z_p))$ or, equivalently, as a function on $\GL_2(\Q_p)$ satisfying $\phi(\gamma b\sigma)=\sigma^{-1}\phi(b)$, for all $\gamma\in \iota_p(R^\times)$, $b\in\GL_2(\Q_p)$ and $\sigma\in\Gamma_0(p\Z_p)$.

For any integer $n\geq 0$, we still use the symbol $\mathcal{P}_n$ for  the $\Q_p$-vector space of homogeneous polynomials in two variables of degree $n$, and the same for the dual space $V_n$. If $k=n-2$, the space $S(\Sigma,V_n)$ is referred to as the space of weight $k$ automorphic forms on $B$ of level $\Sigma$, and it is denoted by $S_k(\Sigma)$. Fix $U\subseteq\mathcal{W}$ a neighborhood of $k_0$. Set $L_*=\Z_p^2$. 
For every integer $k\geq 2$ in $U$, there exists a specialization map
$$\rho_k\colon S\left(\Sigma_\infty,\mathcal{D}(L_*',U)\right)\longrightarrow S_k(\Sigma)$$
defined by
$$(\rho_k(\Phi)(g))(P):=\int_{\Z_p^\times\times p\Z_p}P(x,y)d\Phi(g),$$
for all $g\in\GL_2(\Q_p)$ and $P\in\mathcal{P}_n$, where $n=k-2$. 

Let $\varphi_{f}\in S_{k_0}(\Sigma(N^+p,N^-))$ be the modular form 
corresponding to $f$ via the Jacquet-Langlands correspondence, normalised as in \cite[\S3.2]{Sev}. By \cite[Theorem 3.7]{Sev} (see also \cite{LV}) there exists a connected neighborhood $U\subseteq\mathcal{W}$ of $k_0$ 
and 
\begin{equation}\label{Phi}
\Phi\in S\left(\Sigma_\infty,\mathcal{D}(L_*',U)\right)\end{equation} such that
$\rho_{k_0}(\Phi)=\varphi_f.$

\subsection{Semidefinite integrals and generalised Heegner cycles}
Choose the branch of the $p$-adic logarithm $\log_f$ as in \cite[\S5.2]{Sev}. 
Recall the element $\Phi$ in \eqref{Phi}. 
Out of $\Phi$, one constructs as explained in \cite[Proposition 3.5]{Sev}, a collection of measure $\{\mu_{L}\}_L$ with $\mu_L\in\mathcal{D}(L',U)$ indexed by lattices $L$ of $\Q_{p}^2$. 


For the next definition of \emph{semidefinite integral}, which  
can be found in \cite[Section 5.2]{Sev}, we use the following notation: for any point $z\in\HH_p(\Q_{p^2})$ whose reduction to the special fiber is non-singular, we denote 
$L_z$ the lattice associated 
with the reduction of $z$ and $|L_z|$ its $p$-adic size; see 
\cite[page 115]{Sev}, to which the reader is referred to for details. 

\begin{definition} The \emph{semidefinite integral} is the 
function 
\begin{equation}\label{semidefinite}
(z,Q)\longmapsto{\int^z Q\omega_f}:=\frac{1}{|L_{z}|^{m_0}}\frac{d}{dk}\left(\int_{L_z'}Q(x,y)\langle x-z y\rangle^{k-k_0}d\mu_{L_z}(x,y)\right)_{\vert_{k=k_0}}\end{equation}
defined for $Q\in \mathcal{P}_{n_0}$ and $z\in\HH_p(\Q_{p^2})$ whose reduction to the special fiber is non-singular. \end{definition}

We now connect semidefinite integrals and generalised Heegner cycles. 
For each $Q\in \mathcal{P}_{n_0}$, denote $Q^\vee$ the element in $V_{n_0}$ defined by 
$Q^\vee(P)=\langle Q,P\rangle_{\mathcal{P}_{n_0}}$ for $P\in\mathcal{P}_{n_0}$. For a fixed $z\in\mathcal{H}_p(\Q_p^\mathrm{unr})$ define the following element of $V_{n_0}$:
$$Q\longmapsto\langle F_{\omega_f}(z),Q^\vee\rangle_{V_{n_0}}$$
where we identify as above $(\mathcal{V}_{n_0})_{z}$ with $V_{n_0}$; 
recall that $F_{\omega_f}$ is the Coleman primitive of $\omega_f$.  

\begin{lemma}\label{lemma7.3}
One has
\begin{enumerate}
\item $\langle F_{\omega_f}(\gamma(z)),Q^\vee\rangle_{V_{n_0}}=\langle F_{\omega_f}(z),(Q\cdot\gamma)^\vee\rangle_{V_{n_0}}$, for every $\gamma \in \Gamma$;
\item $\langle F_{\omega_f}(z_2),Q^\vee\rangle_{V_{n_0}}-\langle F_{\omega_f}(z_1),Q^\vee\rangle_{V_{n_0}}=\int_{z_1}^{z_2}f(z)Q(z)dz$.
\end{enumerate}
\end{lemma}

\begin{proof} The second statement is a consequence of \eqref{omega_f} and the definition of Coleman primitive, since 
$$d\langle F_{\omega_f}(z),Q^\vee\rangle=f(z)\langle\partial^0,Q^\vee\rangle dz=f(z)Q(z)dz.$$ 
We need to prove (1). 
Since $f$ has level $\Gamma$, its Coleman primitive $F_{\omega_f}$ is $\Gamma$-invariant, i.e. $\gamma F_{\omega_f}=F_{\omega_f}$ for every $\gamma\in\Gamma$, where $(\gamma F_{\omega_f})(z):=\gamma F_{\omega_f}(\gamma^{-1}z)$ (note that the action on the right hand side is the one on $V_n$). This means that $F_{\omega_f}(\gamma(z))=\gamma F_{\omega_f}(z)$ for every $\gamma\in\Gamma$. 
Recall that $\langle Av_1,v_2\rangle_{V_n}=\langle v_1,\bar{A}v_2\rangle_{V_n}$; thus, for every $\gamma\in\Gamma$ we have 
\[\langle \gamma F_{\omega_f}(z),Q^\vee\rangle_{V_{n_0}}=
\langle  F_{\omega_f}(z),\gamma^{-1}Q^\vee\rangle_{V_{n_0}}=
\langle  F_{\omega_f}(z),(Q\cdot \gamma)^\vee\rangle_{V_{n_0}}\] which proves (1). 
\end{proof}

\begin{theorem}\label{theorem semidefinite} Let $\varphi:A_0\rightarrow A$ be an isogeny and $Q^\vee=v$ for some $v\in (\mathcal{V}_{n_0})_{z_A}$. Then 
\[\deg(\varphi)\cdot\int^{z_A} Q\omega_f=\AJ_p(\Delta_\varphi)(f\otimes \varphi^\vee_*(v)).\]
\end{theorem}

\begin{proof}
By \cite[Lemma 5.6]{Sev}, there is a unique function $(z,Q)\mapsto F(z,Q)$ for $z\in\mathcal{H}(\Q_{p^2})$, $Q\in\mathcal{P}_{n_0}$ 
satisfying the 
following properties: 
\begin{enumerate}
\item $F(\gamma(z),Q)=F(z,Q\cdot\gamma)$, 
\item $F(z_1,Q)-F(z_2,Q)=\int_{z_2}^{z_1}f(z)Q(z)dz$,
\end{enumerate} 
for all $z$, $z_1$, $z_2$ and all $Q$. 
By Lemma \ref{lemma7.3} we have 
\begin{equation}\label{equiv}
\int^{z_A} Q\omega_f=\langle F_{\omega_f}(z_A),Q^\vee\rangle_{V_{n_0}}.
\end{equation} The result follows then from Corollary \ref{coro5.6}. 
\end{proof}

\subsection{Heegner points, optimal embeddings and false isogenies}\label{Heegner section}

A \emph{Heegner point} (of conductor $1$) on the Shimura curve $X=X_{N^+,pN^-}$ is a point on $X$ corresponding to an abelian surface $A$ with quaternionic multiplication and level $N^+$ structure, such that the ring of endomorphisms of $A$ (over an algebraic closure of $\Q$) which commute with the quaternionic action and respect the level $N^+$ structure is isomorphic to $\mathcal{O}_K$. 
The theory of complex multiplication implies that they are all defined over the Hilbert class field $H$ of $K$. We denote
$\mathrm{Heeg}(\mathcal{O}_K)$ denotes the set of 
Heegner points of conductor $1$ on $X$. 

We now recall Shimura reciprocity law, referring to \cite[\S2.5]{Brooks} for details. 
Fix an ideal $\mathfrak{a}\subseteq\mathcal{O}_K$ and an Heegner point $z$. 
We have then an embedding $\iota_z:K\hookrightarrow\mathcal{B}$,
and since the class number of the indefinite quaternion algebra $\mathcal{B}$ is equal to $1$, there is $\alpha\in\mathcal{B}$ such that 
$\iota_z(\mathfrak{a})\mathcal{R}_\mathrm{max}=\alpha\mathcal{R}_\mathrm{max}$. Right multiplication by $\alpha$ gives a false isogeny 
$\varphi_\alpha:A_z\rightarrow A_{\alpha(z)}$, where for any point $x\in X$ we let $A_x$ denote the false elliptic curve corresponding to $x$. 
If $(\mathfrak{a},N^+M)=1$ then this is a false isogeny of degree prime to $N^+M$. Since $\alpha(z)$ only depends on $\mathfrak{a}$ and not on the choice of $\alpha$, we may write $\alpha(z)=\mathfrak{a}\star z$,  $A_{\mathfrak{a}\star z}=A_{\alpha(z)}$ and $\varphi_\mathfrak{a}=\varphi_\alpha$. 
If we denote $\sigma_\mathfrak{a}$ the element in $\Gal(H/K)$ corresponding to $\mathfrak{a}$ via the arithmetically normalized Artin reciprocity map, Shimura reciprocity law shows that 
$\sigma_\mathfrak{a}(z)=\mathfrak{a}\star z$. Moreover, if we denote 
$\mathcal{W}$ the group of Atkin-Lehner involutions acting on $X$, 
the action of $\mathcal{W}\times\Gal(H/K)$ on the set $\mathrm{Heeg}(\mathcal{O}_K)$ is simply transitive (see \cite[\S 2.3]{BD} or \cite[page 366]{IS}). Fixed a point $z_{0}$ corresponding to the false elliptic curve $A_0$, the correspondences 
$\mathfrak{a}\mapsto \mathfrak{a}\star z_{0}$ and $\mathfrak{a}\mapsto\varphi_\mathfrak{a}:A_0=A_{z_0}\rightarrow A_{\mathfrak{a}\star z_0}$ set up a bijection 
\begin{equation}\label{bije2}
\mathrm{Heeg}(\mathcal{O}_K)\longleftrightarrow
\mathrm{Isog}(A_0)\end{equation}
where $\mathrm{Isog}(A_0)$ denotes the set of false isogenies $\varphi:A_0\rightarrow A$ of degree prime to $N^+M$.

An embedding of $\Q$-algebras $\Psi\colon K\rightarrow B$ is called \emph{optimal of level $N^+$} if $\Psi^{-1}(R)=\cO_K[1/p]$. The group $\Gamma$ acts by conjugation on the set of optimal embeddings. 
Let $\mathrm{Emb}(\mathcal{O}_K)$ be the set of $\Gamma$-conjugacy classes of optimal embeddings, which is non-empty under our assumption (see \cite[Lemma 2.1]{BD-Heegner}).
By \cite[Theorem 5.3]{BD-CD} there exists a bijection 
\begin{equation}\label{bije}
\mathrm{Heeg}(\mathcal{O}_K)\longleftrightarrow \mathrm{Emb}(\mathcal{O}_K).
\end{equation}
We briefly describe how this bijection is obtained. 
Let $z_A$ be an Heegner point corresponding to the abelian surface $A$. 
Let $\End(A)$ denote the endomorphism rings 
of $A$ and $\End(\bar{A})$ the endomorphism rings of the reduction 
$\bar{A}$ of the abelian varietiy $A$ modulo $p$. Define
$\End^0(A)=\End(A)\otimes_\Z\Q$ and 
$\End^0(\bar{A})=\End(\bar{A})\otimes_\Z\Q$ and let 
$\End^0_\mathcal{B}(A)$ and 
$\End^0_\mathcal{B}(\bar{A})$  denote the endomorphisms 
which commute with the action of the quaternion algebra $\mathcal{B}$. 
We then have $\End^0_\mathcal{B}(A)\simeq K$ 
and $\End^0_\mathcal{B}(\bar{A})\simeq B$, and 
the map $\Psi$ associated with $z_A$ as in \eqref{bije} is  
the reduction of endomorphisms:
\[\Psi_{A}:K=\End_\mathcal{B}^0(A)\longrightarrow \End_\mathcal{B}^0(\bar{A})=B.\]
On the other hand, let $\Psi\colon K\longrightarrow B$ be an optimal embedding of level $N^+$. It determines a local embedding $\Psi\colon K_p\longrightarrow B_p$ which we denote in the same way by an abuse of notation. The local embedding $\Psi$ defines an action of $K_p^\times$ on $\HH_p(K_p)$ which has two fixed points, $z_\Psi$ and $\bar{z}_\Psi$. The Heegner point associated to $\Psi$ by \eqref{bije} is the point on $X$ corresponding via the Cerednik-Drinfeld uniformization to the class modulo $\Gamma$ of $z_\Psi$. Abusing notation, in the following we will use the symbol $z_\Psi$ to denote both the fixed point in $\HH_p(K_p)$ and its class in $\Gamma\backslash \HH_p(K_p)=X(K_p)$.


In light of the previous paragraphs, given $\varphi:A_0\rightarrow A\in\mathrm{Isog}(A_0)$, we denote $z_\varphi$ the Heegner points corresponding to $\varphi$ by \eqref{bije2} 
and $\Psi_\varphi$ the optimal embedding 
corresponding to $z_\varphi$ by \eqref{bije}. 
For $\varphi$ the identity map, we denote $z_\varphi$ by $z_0$ 
and $\Psi_\varphi$ by $\Psi_0$.  
Moreover, if we start with an optimal embedding 
$\Psi$, we denote $z_\Psi$ the Heegner point corresponding to $\Psi$ by \eqref{bije}
and $\varphi_\Psi$ the false isogeny corresponding to $z_\Psi$ by \eqref{bije2}.  Finally, if we start with an Heegner point $z$, we denote $\Psi_z$ the optimal embedding corresponding to $z$ via \eqref{bije} 
and $\varphi_z:A_0\rightarrow A_z$ the false isogeny corresponding to 
$z$ via \eqref{bije2}. 
We also introduce a convention for the Galois action: 
for any $\sigma=\sigma_\mathfrak{a}\in\Gal(H/K)$, 
we denote $z_A^\sigma=\mathfrak{a}\star z_A$,
$A^\sigma=A_{\sigma(z_A)}$ and $\Psi^\sigma=\Psi_{z_A^\sigma}$.

Denote $z\mapsto\bar{z}$ the action of the non-trivial automorphism 
$c\in\Gal(\Q_{p^2}/\Q_p)$ on $\mathcal{H}_p(\Q_{p^2})$. 
For each optimal embedding $\Psi$, denote 
\[P_\Psi(x,y)=cx^2+(d-a)xy-by^2=A_\Psi(x-z_\Psi y)(x-\bar{z}_\Psi y)\] the polynomial associated to $\Psi$, 
where $\iota_p(\Psi(\sqrt{D}))=\smallmat abcd$ and $D$ is the discriminant of $K$ (\emph{cf.} as in \cite[(84)]{BD}).
Define the the polynomials
\[Q_{\Psi}^{(j)}(x,y)=(x-z_\Psi y)^{m_0+j}(x-\bar{z}_\Psi y)^{m_0-j}\]
for any positive integer $k$ and any integer $j=-n_0/2,\dots,n_0/2$.
Put $v_\Psi^{(j)}=(Q_\Psi^{(j)})^\vee$ 
and define 
\[v_\varphi^{(j)}=\varphi^\vee_*(v_{\Psi_{\varphi}}^{(j)}).\] 


%

\begin{proposition}\label{prop6.6} Let $\varphi:A_0\rightarrow A$ be a false isogeny. Then 
\[\deg(\varphi)\cdot \int^{z_\varphi} Q_{\Psi_\varphi}^{(j)}\omega_f=\AJ_p(\Delta_\varphi)\left(f\otimes v_\varphi^{(j)}\right).\]
\end{proposition}
\begin{proof} Since $\varphi_*(v_\varphi^{(j)})=\deg(\varphi)\cdot v_{\Psi_\varphi}^{(j)}$, 
the proposition follows from Theorem \ref{theorem semidefinite}. 
\end{proof}

Let $\varphi:A_0\rightarrow A$ be a false isogeny. The abelian variety 
$A$ is defined over $H$, and therefore it is also defined over $\Q_{p^2}$, 
because $p$ is inert in $K$ and therefore splits completely in $H$. 
Let $\bar{A}$ denote the abelian variety obtained by applying to $A$ 
the non-trivial automorphism $c$ of $\Gal(\Q_{p^2}/\Q_p)$, and still denote 
$c:A\rightarrow \bar{A}$ the map induced by $c$. If $\varphi:A_0\rightarrow A$ is a false isogeny, 
then we denote $\bar{\varphi}=c\circ\varphi:A_0\rightarrow\bar{A}$ the isogeny obtained by composition $\varphi$ with $c:A\rightarrow\bar{A}$.
Let $W_p:X\rightarrow X$ denote the Atkin-Lehner involution at $p$. 
If we denote $w_p$ any element of $\mathcal{R}^\times$ such that the $p$-adic valuation of its norm is equal to $1$, which we fix from now on, 
then we have $W_p(z)=w_p(z)$. 
We have (see \emph{e.g.} \cite[Theorem 4.7]{BD-CD}) 
\begin{equation}\label{AL}
z_{\bar{A}}=w_p(\bar{z}_A).
\end{equation}
For the next result, define 
\[ \bar{v}_\varphi^{(j)}= \bar{\varphi}_*^\vee\left(\left(Q_{\Psi_\varphi}^{(j)}|w_p\right)^\vee\right).\]


\begin{proposition}\label{prop6.5} Let $\varphi:A_0\rightarrow A$ be an isogeny. Then 
\[\deg(\varphi)\cdot
\int^{\bar{z}_\varphi} Q_{\Psi_\varphi}^{(j)}\omega_f=\omega_p\cdot\AJ_p(\Delta_{\bar{\varphi}})(f\otimes \bar{v}_\varphi^{(j)}),\]
where $\omega_p\in\{\pm 1\}$ is the eigenvalue of the Atkin-Lehner involution at $p$ acting on $f$.
\end{proposition}

\begin{proof} 
By \eqref{AL}, and the fact that $W_p$ is an involution, we have 
\[\int^{\bar{z}_\varphi} Q_{\Psi_\varphi}^{(j)}\omega_f=\int^{w_p(z_{\bar{\varphi}} )}Q_{\Psi_\varphi}^{(j)}\omega_f.\]
Since $W_p$ acts on $F_{\omega_f}$ as multiplication by $\omega_p\in\{\pm 1\}$, one easily checks (using the same calculations as in Lemma \ref{lemma7.3}) that 
\[\int^{w_p(z_{\bar{\varphi}} )}Q_{\Psi_\varphi}^{(j)}\omega_f=
\omega_p\cdot\int^{z_{\bar{\varphi}}}\left(Q_{\Psi_\varphi}^{(j)}|w_p\right)\omega_f.\]
The result follows then from Theorem 
\ref{theorem semidefinite}. 
\end{proof}

\subsection{Two variables anticyclotomic $p$-adic $L$-functions}
For each optimal embedding $\Psi$, we consider the lattice $L_\Psi=L_{z_\Psi}$; recall that this lattice is characterised up to homothety by the condition that $L_\Psi$ is stable by the action of $\Psi(K_p^\times)$, where $K_p=K\otimes_\Q\Q_p\simeq\Q_{p^2}$ (\emph{cf.} \cite[\S3.2]{BD}). 
Recall that the function $(x,y)\mapsto\ord_p\left(P_\Psi(x,y)\right)$ is constant 
on $L'_\Psi$, and its constant value is equal to $\ord_p(|L_\Psi|)$ (see \cite[Lemma 3.7]{BD}). 
Therefore, by eventually translating $(\Psi,L_\Psi)$ by an appropriate element of $R^\times$ in such a way that $|L_\Psi|=1$, we have 
$\langle P_\Psi(x,y)\rangle=P_\Psi(x,y)$ for all $(x,y)\in L_\Psi'$.
Moreover, note that 
\[Q_{\Psi}^{(j)}(x,y)=
\frac{P_\Psi(x,y)^{m_0}}{A_\Psi^{m_0}}\left(\frac{x-z_\Psi y}{x-\bar{z}_\Psi y}\right)^j.
\] If $j\equiv 0\pmod{p+1}$, then we have 
\[\left(\frac{x-z_\Psi y}{x-\bar{z}_\Psi y}\right)^j=\left\langle\frac{x-z_\Psi y}{x-\bar{z}_\Psi y}\right\rangle^j\]
for all $(x,y)\in L'_\Psi$. In fact, the $p$-adic valuation of $x-z_\Psi y$ and $x-\bar{z}_\Psi y$ are equal and, if $x-z_\Psi y=\zeta\langle x-z_\Psi y\rangle$ then $x-\bar{z}_\Psi y=\bar{\zeta}\langle x-\bar{z}_\Psi y\rangle$, where $\zeta$ is a $(p^2-1)$-th root of unity. Since $\frac{\zeta}{\bar{\zeta}}=\frac{\zeta}{\zeta^p}=\zeta^{p^2-p}$, if $j\equiv 0\pmod{p+1}$ then $\zeta^{j(p^2-p)}=1$.

\begin{definition}\label{def of p-adic L-function 1}
The \emph{partial two-variable anticyclotomic $p$-adic $L$-function}
associated to $\Phi$ and $[\Psi]\in \mathrm{Emb}(\cO_K)$ 
is the function defined for $(k,s)\in U\times\Z_p$ as
\[\mathcal{L}_p(\Phi/K,\Psi,k,s)=\frac{A_\Psi^{\frac{k-k_0}{2}}}{|L_\Psi|^{m_0}}
\int_{L'_\Psi}P_\Psi^{m_0}(x,y)\langle x-z_\Psi y\rangle^{s-k_0/2}\langle x-\bar{z}_{\Psi}y\rangle ^{k-s-k_0/2}d\mu_{L_\Psi}.\]
\end{definition}

The restriction of $\mathcal{L}_p(\Phi/K,\Psi,k,s)$ to the line $s=k/2+j$, for $-n/2\leq j\leq n/2$ an integer, is then the function 
\[\mathcal{L}_p^{(j)}(\Phi/K,\Psi,k)=\frac{A_\Psi^{\frac{k-k_0}{2}}}{|L_\Psi|^{m_0}}
\int_{L'_\Psi}P_\Psi^{m_0}(x,y)
\left\langle 
\frac{x-z_\Psi y}{x-\bar{z}_\Psi y}
\right\rangle^j
\langle x-z_\Psi y\rangle^{\frac{k-k_0}{2}}\langle x-\bar{z}_{\Psi}y\rangle ^{\frac{k-k_0}{2}}d\mu_{L_\Psi}.\]
 
\begin{proposition}\label{prop7.5} Let $\varphi:A_0\rightarrow A$ be a false  isogeny.  
Suppose that 
$j\equiv 0\pmod{p+1}$. Then we have $\mathcal{L}_p^{(j)}(\Phi/K,\Psi_\varphi,k_0)=0$ and 
\[\frac{d}{dk}\left(\mathcal{L}_p^{(j)}(\Phi/K,\Psi_\varphi,k)\right)_{\vert_{k=k_0}}=
\frac{A_{\Psi_\varphi}^{m_0}}{2\deg(\varphi)}\left(\AJ_p(\Delta_\varphi)(f\otimes v_{\varphi}^{(j)})+
\omega_p\cdot
\AJ_p(\Delta_{\bar{\varphi}})(f\otimes \bar{v}_\varphi^{(j)})\right)
.\]
\end{proposition}

\begin{proof} The congruence conditions imposed to $j$
combined with the observations before Definition \ref{def of p-adic L-function 1} 
imply that 
\[\mathcal{L}_p^{(j)}(\Phi/K,\Psi_\varphi,k)=\frac{A_{\Psi_\varphi}^{\frac{k-k_0}{2}}}{|L_{\Psi_\varphi}|^{m_0}}
\int_{L'_{\Psi_\varphi}}A_{\Psi_\varphi}^{m_0}Q_{\Psi_\varphi}^{(j)}(x,y)
\langle x-z_\varphi y\rangle^{\frac{k-k_0}{2}}\langle x-\bar{z}_\varphi y\rangle ^{\frac{k-k_0}{2}}d\mu_{L_{\Psi_\varphi}}.\]
The value at $k_0$ is then 
\[\mathcal{L}_p^{(j)}(\Phi/K,\Psi_\varphi,k_0)=\frac{A_{\Psi_\varphi}^{m_0}}{|L_{\Psi_\varphi}|^{m_0}}
\int_{L'_{\Psi_\varphi}}Q_{\Psi_\varphi}^{(j)}(x,y)d\mu_{L_{\Psi_\varphi}}\]
which is equal to $0$ by \cite[Propositions 3.8 and 6.2]{Sev}. 
By \cite[Proposition 3.1]{Sev}, for any $Q\in\mathcal{P}_{n_0}$, any lattice $L$ and any $z_1,z_2\in \mathcal{H}_p(\Q_{p^2})$  
we have
\[
\frac{d}{dk}\left(
\int_{L'}Q(x,y)\langle x-z_1 y\rangle^{\frac{k-k_0}{2}}
\langle x-z_2 y\rangle^{\frac{k-k_0}{2}}d\mu_L\right)_{\vert k=k_0}
\]
is the sum
\[\frac{1}{2}\frac{d}{dk}\left(\int_{L'}Q(x,y)\langle x-z_1 y\rangle^{k-k_0}d\mu_L\right)_{\vert k=k_0}+
\frac{1}{2}\frac{d}{dk}\left(\int_{L'}Q(x,y)\langle x-z_2 y\rangle^{k-k_0}d\mu_L\right)_{\vert k=k_0}.\]
If we take $L=L_{z_\varphi}=L_{\Psi_\varphi}$, $z_1=z_\varphi$ and $z_2=\bar{z}_\varphi$, 
the first summand in the 
above formula is $\frac{1}{2}|L_{\Psi_\varphi}|^{m_0}\int^{z_\varphi}Q\omega_f$, while the second summand is 
\[\frac{1}{2}\frac{d}{dk}\left(\int_{L_{z_\varphi}}Q(x,y)\langle x-\bar{z}_\varphi y\rangle^{k-k_0}d\mu_{L_{z_\varphi}}\right)_{\vert k=k_0}.\]
We now observe that $L_{z_\varphi}=L_{\bar{z}_\varphi}$ and therefore the second summand is $\frac{1}{2}|L_{\Psi_\varphi}|^{m_0}\int^{\bar{z}_\varphi} Q\omega_f$: this is because, as recalled above, lattice $L_\Psi$ attached to an optimal embedding $\Psi:K\rightarrow B$ is characterised up to homothety by the condition that $L_\Psi$ is stable by the action of $\Psi(\Q_{p^2})$. 
The result then follows from  
Proposition \ref{prop6.6} and Proposition \ref{prop6.5}.\end{proof} 

Let $K_\infty$ be the maximal anticyclotomic extension of $K$ which is unramified outside $p$. Write $\tilde{G}$ for $\Gal(K_\infty/K)$ and $\Delta$ for $\Gal(H/K)$. As recalled above, 
the group $\mathcal{W}\times\Delta$ acts freely and transitively on $\Emb(\mathcal{O}_K)$, and, by the Shimura Reciprocity Law, 
this action corresponds to the natural action of $\mathcal{W}\times\Delta$ on the set of Heegner points under the bijection \eqref{bije}. 
Denote by $\Xi$ the set of $\Delta$-orbits in $\mathrm{Emb}(\cO_K)$ and fix $\xi\in\Xi$. If $\Delta=\{\bar{\delta}_1,\dots,\bar{\delta}_h\}$ then $\Psi_i=\Psi_0^{\bar{\delta}_i^{-1}}$ are representatives for the elements of $\Xi$, for a fixed $\Psi_0\in\mathrm{Emb}(\cO_K)$.
Let $\chi\colon \tilde{G}\rightarrow\overline{\Q}_p^\times$ be a character factoring through $\Delta$. 
The optimal embeddings $\Psi_i$ correspond to Heegner points $z_i=z_0^{\bar{\delta}_i^{-1}}$, 
and these come from isogenies $\varphi_i=\varphi_0^{\bar{\delta}_{i}^{-1}}:A_0\rightarrow A_i=A_{z_i}$. 

\begin{definition}\label{def of p-adic L-function 2} The \emph{two-variable anticyclotomic $p$-adic $L$-function} associated to $\Phi$ and the character $\chi$
is the function defined for $(k,s)\in U\times\Z_p$ as
$$\mathcal{L}_p(\Phi/K,k,s,\chi)=\sum_{i=1}^h\chi(\delta_i)\cdot\mathcal{L}_p(\Phi/K,\Psi_i,k,s),$$
where $\delta_i\in\tilde{G}$ is a lift of $\bar{\delta}_i\in\Delta$.
\end{definition}
The restriction of $\mathcal{L}_p(\Phi/K,\Psi,k,s)$ to the line $s=k/2+j$, for $-n/2\leq j\leq n/2$ an integer, is then the function 
\[
\mathcal{L}_p^{(j)}(\Phi/K,k,\chi)=\sum_{i=1}^h\chi(\delta_i)\cdot\mathcal{L}_p^{(j)}(\Phi/K,\Psi_i,k).\]

\begin{theorem}\label{TheoremB} Suppose that 
$j\equiv 0\pmod{p+1}$. Then we have $\mathcal{L}_p^{(j)}(\Phi/K,k,\chi)=0$
and 
\[\frac{d}{dk}\left(\mathcal{L}_p^{(j)}(\Phi/K,k,\chi)\right)_{\vert k=k_0}=\sum_{i=1}^h
\frac{A_{\Psi_i}^{m_0}\cdot\chi(\delta_i)}{2\cdot \deg(\varphi_i)}
\left(\AJ_p(\Delta_{\varphi_i})(f\otimes v_{\varphi_i}^{(j)})+\omega_p
\AJ_p(\Delta_{\bar{\varphi}_i})(f\otimes \bar{v}_{{\varphi_i}}^{(j)})
\right).\]
\end{theorem}

\begin{proof} 
The result follows from Proposition \ref{prop7.5} and the definitions. 
\end{proof}

\begin{remark}
The function $\mathcal{L}_p(\Phi/K,k,\chi)=\mathcal{L}_p^{(0)}(\Phi/K,k,\chi)$ is a \emph{square-root} $p$-adic $L$-function, in the sense that the value of 
\[{L}_p(\Phi/K,k,\chi)=\mathcal{L}_p(\Phi/K,k,\chi)\cdot \mathcal{L}_p(\Phi/K,k,\chi^{-1})\] at integers $k\geq 2$, $k\equiv k_0 \pmod{p-1}$, $k\neq k_0$, satisfies an 
interpolation formula of the following shape: 
\[L_p(\Phi/K,k,\chi)\overset\cdot=L_K^\mathrm{alg}(f_k^\sharp,\chi,k/2).\]
In the formula above we adopt the following notation. 
First, for each even integer $k$ as above, 
let $f_k$ be the  classical modular form of $\Gamma_0(N)$ and weight $k$ which 
correspond under the Jacquet-Langlands correspondence to the specialization 
$\rho_k(\Phi)\in S_k(\Sigma)$ of $\Phi$ in weight $k$, it is well defined up to scalars; a version for families of the Jacquet-Langlands correspondence allows us to see these forms as classical specialisations of a Coleman family $f_\infty$ of modular forms. Denote  
$f_k^\sharp$ the newform of level $N/p$ whose $p$-stabilisation is 
$f_k$ if $k\neq k_0$ or $f$ is old at $p$, and $f_{k_0}^\sharp=f$ otherwise; 
 $L_K^\mathrm{alg}(f_k^\sharp,\chi,k/2)$ denote the algebraic part of the 
value at $s=k/2$ of the complex $L$-function $L_K(f_k^\sharp,\chi,s)$, which is obtained by dividing 
$L_K(f_k^\sharp,\chi,k/2)$ by a suitable complex period; the symbol $\overset\cdot=$ means that the equality is up to explicit algebraic factors. 
See \cite[Theorem 9.1]{Sev} for details. 
It is a very interesting task to investigate similar interpolation properties of
$\mathcal{L}_p^{(j)}(\Phi/K,k,\chi)$: the natural question is 
if $\mathcal{L}_p^{(j)}(\Phi/K,k,\chi)$ is related to $L_K^\mathrm{alg}(f_k,\chi,k/2+j)$ in a way similar to what happens in the case $j=0$. 
\end{remark}

\subsection{One variable anticyclotomic $p$-adic $L$-functions}
In this section we use the results collected in the previous sections to give an extension of the results in \cite{Masdeu} on the first derivative of the $1$-variable anticyclotomic $p$-adic $L$-function. 

Denote by $L_p(f/K,\Psi,\star,s)$ the \emph{partial anticyclotomic $p$-adic $L$-function} of $f$ and $K$ attached to the pair $(\Psi,\star)$, where $\Psi$ is an optimal embedding as in \S\ref{Heegner section} 
and $\star\in\PP^1(\Q_p)$ a base point (\cite{BDIS}); this is a function of the $p$-adic variable $s\in\Z_p$ defined by
$$L_p(f/K,\Psi,\star,s)=\int_G \langle\alpha\rangle^{s-\frac{k_0}{2}}d\mu_{f,\Psi,\star}(\alpha),$$
where $\langle\alpha\rangle^t=\exp(t\log_f(\langle\alpha\rangle))$ for all $t\in\Z_p$ and $\mu_{f,\Psi,\star}$ is the local analytic distribution on $G=K_{p,1}^\times$, the compact subgroup of $K_p^\times$ of elements of norm $1$, defined in \cite[Section 2.4]{BDIS}. 


\begin{proposition}\label{prop6.1} Let $\varphi:A_0\rightarrow A$ be a false isogeny.  
For integer $-n_0/2\leq j\leq n_0/2$ with $j\equiv0\pmod{p+1}$ we have
$L_p(f/K,\Psi_\varphi,\infty,k_0/2+j)=0$ and 
\[L_p'(f/K,\Psi_\varphi,\infty,s)_{\vert s=\frac{k_0}{2}+j}=\frac{A_{\Psi_\varphi}^{m_0}}{\deg(\varphi)}\left(\AJ_p(\Delta_\varphi)(f\otimes v_{\varphi}^{(j)})-\omega_p\cdot
\AJ_p(\Delta_{\bar{\varphi}})(f\otimes \bar{v}_{\varphi}^{(j)})
\right).\]\end{proposition}

\begin{proof} We sketch the proof, following closely \cite[Theorem 5.3]{Masdeu} (but see Remark \ref{rem}). 
Thanks to the congruence conditions imposed to $j$, have 
\[
L_p(f/K,\Psi_\varphi,\infty,k_0/2+j)=\int_G \alpha^{j}d\mu_{f,\Psi_\varphi,\infty}(\alpha),\]
where now $\alpha^j$ is the usual $j$-fold product of $\alpha$ by itself, 
and therefore the above integral vanishes thanks to \cite[Lemma 5.1]{Masdeu}. 
For the value of the derivative, we begin by observing that, thanks to 
the congruence conditions imposed to $j$, we have $\langle\alpha\rangle^j=\alpha^j$, and therefore 
\[L_p'(f/K,\Psi_\varphi,\infty,s)_{\vert s=\frac{k_0}{2}+j}= 
\int_G\log_f(\langle\alpha\rangle)\langle\alpha\rangle^jd\mu_{f,\Psi_\varphi,\infty}(\alpha).
\] 
Let now $\mu_f$ the measure on $\PP^1(\Q_p)$ attached to $f$ in \cite[Proposition 9]{Tei} using the 
harmonic cocycle attached to $f$. Then we have 
\[\begin{split}
\int_G\log_f(\langle\alpha\rangle)\langle\alpha\rangle^jd\mu_{f,\Psi_\varphi,\infty}(\alpha)&=\int_{\PP^1(\Q_p)}\log_f\left(\frac{x-z_\varphi}{x-\bar{z}_\varphi}\right)
\cdot\left(\frac{x-z_\varphi}{x-\bar{z}_\varphi}\right)^jP_{\Psi_\varphi}^{m_0}(x)d\mu_f(x)\\
&=\int_{\PP^1(\Q_p)}\left(\int_{\bar{z}_\varphi}^{z_\varphi}\frac{dz}{z-x}\right)
\cdot\left(\frac{x-z_\varphi}{x-\bar{z}_{\varphi}}\right)^jP_{\Psi_\varphi}^{m_0}(x)d\mu_f(x)\\
&= 
\int_{\bar{z}_\varphi}^{z_\varphi}\left(\int_{\PP^1(\Q_p)}\frac{1}{z-x}
\cdot\left(\frac{x-z_\varphi}{x-\bar{z}_\varphi}\right)^j P_{\Psi_\varphi}^{m_0}(x)d\mu_f(x)\right)dz\\
&= 
\int_{\bar{z}_\varphi}^{z_\varphi}\left(\int_{\PP^1(\Q_p)}\frac{d\mu_f(x)}{z-x}\right)
\cdot\left(\frac{z-z_\varphi}{z-\bar{z}_\varphi}\right)^jP_{\Psi_\varphi}^{m_0}(z)dz\\
&=\int_{\bar{z}_\varphi}^{z_\varphi}
f(z)\left(\frac{z-z_\varphi}{z-\bar{z}_\varphi}\right)^j P_{\Psi_\varphi}^{m_0}(z)dz
\end{split}\]
where the first equality follows from the definition of the $p$-adic $L$-function 
in \cite[\S 2.4]{BDIS}, the second equality follows from the definition of Coleman 
integral, the third follows from the fact that we can reverse the order of integration by applying the reasoning in the proof of Theorem 4 of \cite{Tei}, the fourth from the fact that
$$\int_{\PP^1(\Q_p)}\frac{1}{z-x}
\cdot\left(\frac{x-z_\Psi}{x-\bar{z}_{\Psi}}\right)^j P_\Psi^{m_0}(x)d\mu_f(x)=\int_{\PP^1(\Q_p)}\frac{1}{z-x}
\cdot\left(\frac{z-z_\Psi}{z-\bar{z}_{\Psi}}\right)^j P_\Psi^{m_0}(z)d\mu_f(x),$$
since the two functions inside the integral differ by a polynomial 
of degree at most $n_0$ in $x$, and the last equality follows from Teitelbaum's $p$-adic Poisson inversion formula (we refer to 
the proof of \cite[Theorem 5.3]{Masdeu} and \cite[Theorem 3.5]{BDIS} for details).
Combining the above equations we find:
\[\begin{split}L_p'(f/K,\Psi_\varphi,\infty,s)_{\vert s=\frac{k_0}{2}+j}&=\int_{\bar{z}_\varphi}^{z_\varphi}
f(z)\left(\frac{z-z_\varphi}{z-\bar{z}_\varphi}\right)^j P_{\Psi_\varphi}^{m_0}(z)dz\\
&=A_{\Psi_\varphi}^{m_0}
\int_{\bar{z}_\varphi}^{z_\varphi}
f(z)(z-z_\varphi)^{m_0+j}(z-\bar{z}_\varphi)^{m_0-j}dz\\
&=A_{\Psi_\varphi}^{m_0}
\int_{\bar{z}_{\varphi}}^{z_{\varphi}}
f(z)Q_{\Psi_\varphi}^{(j)}dz\\
&=A_{\Psi_\varphi}^{m_0}\left(\int^{z_\varphi}Q_{\Psi_\varphi}^{(j)}\omega_f-\int^{\bar{z}_\varphi}Q_{\Psi_\varphi}^{(j)}\omega_f\right).\end{split}
\]
The result follows then from Propositions \ref{prop6.6} and Proposition \ref{prop6.5}.\end{proof}

\begin{remark}\label{rem}
It seem to the authors that 
\cite[Theorem 5.3]{Masdeu} only works under the congruence condition, $j\equiv0\pmod{p+1}$. In the general case we
have the equality 
\[
L_p(f/K,\Psi_\varphi,\infty,k_0/2+j)=\int_G \langle\alpha\rangle^{j}d\mu_{f,\Psi_\varphi,\infty}(\alpha),\]
where now the function $\alpha\mapsto\langle\alpha\rangle^j$ is locally analytic, 
and is a polynomial only under the congruence conditions on $j$ considered above. 
Therefore, if $j$ does not satisfy the congruence conditions $j\equiv0\pmod{p+1}$ then
one can not directly apply \cite[Lemma 5.1]{Masdeu} to conclude that the value of the 
$p$-adic $L$-function at $k_0/2+j$ vanishes. 
\end{remark}

Recall that we denoted by $K_\infty$ the maximal anticyclotomic extension of $K$ which is unramified outside $p$, by $\tilde{G}$ the Galois group $\Gal(K_\infty/K)$ and by $\Delta$ the Galois group $\Gal(H/K)$. Class field theory implies that the group $G$ can be identified with $\Gal(K_\infty/H)$. Let $\mathrm{Emb}_0(\cO_K)$ be the set of $\Gamma$-conjugacy classes of pairs $(\Psi,\star)$ where $\Psi$ is an optimal embedding and $\star\in\PP_1(\Q_p)$ a base point. The action of $\mathcal{W}\times \Delta$ on $\mathrm{Emb}(\cO_K)$ lifts to a simply transitive action of $\mathcal{W}\times\tilde{G}$ on 
$\mathrm{Emb}_0(\cO_K)$ such that $G$ acts trivially on $\mathrm{Emb}(\cO_K)$. Using this action the distribution $\mu_{f,\Psi,\star}$ on $G$ can be canonically extended to a distribution on $\tilde{G}$ denoted $\mu_{f,K,\xi}$ where $\xi=(\Psi,\star)\in\mathrm{Emb}_0(\cO_K)$ (see \cite[Section 2.5]{BDIS}). 
This distribution depends on the choice of $(\Psi,\star)$ only up to translation by an element of $\tilde{G}$, 
and up to multiplication by $-\omega_p=\pm1$, the negative of the sign of the Atkin-Lehner involution $W_p$ acting on $f$ 
(see \cite[Lemma 2.15]{BDIS}). 

Let $\{\delta_1,\dots,\delta_h\}$ be a set of representatives of the elements of $\Delta$ in $\tilde{G}$, and write 
\[(\Psi_i,\star_i):=\delta_i(\Psi,\star).\] 
Let $\chi\colon \tilde{G}\rightarrow\overline{\Q}_p^\times$ be a continuous character of finite order. We can define the \emph{anti-cyclotomic $p$-adic $L$-function attached to $f$ and $K$ twisted by $\chi$} as
$$L_p(f/K,\xi,\chi,s)=\int_{\tilde{G}}\chi(\alpha)\langle\alpha\rangle^{s-\frac{k_0}{2}}d\mu_{f,K,\xi}(\alpha).$$
If $\chi$ factors through $\Delta$, $L_p(f/K,\xi,\chi,s)$ can be written as a twisted sum of partial $L$-functions
$$L_p(f/K,\xi,\chi,s)=\sum_{i=1}^h \chi(\delta_i)L_p(f/K,\Psi_i,\star_i,s).$$

Since $\mathcal{W}\times\tilde{G}$ acts simply transitively on $\mathrm{Emb}_0(\cO_K)$, for every pair $(\Psi_i,\star_i)$ in the previous sum, there exists a unique $\alpha_i\in\mathcal{W}\times G\subseteq\mathcal{W}\times\tilde{G}$ such that $(\Psi_i,\star_i)=\alpha_i(\Psi_i,\infty)$. If we assume that $\alpha_i\in G=K_{p,1}^\times$, then we have $L_p(f/K,\Psi_i,\star_i,s)=(\alpha_i)^{s-\frac{k_0}{2}}L_p(f/K,\Psi_i,\infty,s)$. We can always do this since the $\Psi_i$'s are in the same $\mathcal{W}$-orbit and, for $w\in\mathcal{W}$
\[ L_p(f/K,w\xi,\chi,s)=\pm L_p(f/K,\xi,\chi,s).\]

Thus, up to sign, we can express the first derivative of the anticyclotomic $p$-adic $L$-function as an explicit combination of values of the Abel-Jacobi images of the cycles $\Delta_{\varphi_i}$. Here $\varphi_i$ denotes the isogeny $A_0\rightarrow A_{\Psi_i}$ associated to $\Psi_i$.

\begin{theorem}\label{TheoremA} Let $\chi:\tilde{G}\rightarrow\bar\Q_p^\times$ be a character factoring through $\Delta$. Then for every integer $j$ such that $-n_0/2\leq j \leq n_0/2$ and $j\equiv0\pmod{p+1}$, we have 
\footnotesize{\[L_p'\left(f/K,\xi,\chi,k_0/2+j\right)=\sum_{i=1}^h\chi(\delta_i)\alpha_i^j
\frac{A_{\Psi_{\varphi_i}}^{m_0}}{\deg(\varphi_i)}\left(\AJ_p(\Delta_{\varphi_i})(f\otimes v_{{\varphi_i}}^{(j)})-\omega_p\cdot
\AJ_p(\Delta_{\bar{{\varphi_i}}})(f\otimes \bar{v}_{{\varphi_i}}^{(j)})
\right).\]}\end{theorem} 

\begin{proof}
This follows directly from the definitions and Proposition \ref{prop6.1}. 
\end{proof}

\begin{remark}
The interpolation properties satisfied by the $p$-adic $L$-function 
$L_p(f/K,\xi,\chi,s)$ and the value of the complex $L$-function $L_K(f,\chi,s)$ at the central critical point $s=k_0/2$ are well-known and carefully discussed in \cite{BDIS}, to which the reader is referred to for details. In particular, in our setting 
both the $p$-adic $L$-function 
and the complex $L$-function 
vanish at $s=k_0/2$. 
It is an interesting task to investigate similar interpolation properties satisfied by the $p$-adic $L$-function $L_p(f/K,\xi,\chi,s)$ and the complex $L$-function $L_K(f_k,\chi,s)$ at 
integers $s=k_0/2+j$ with $n_0/2\leq j\leq n_0/2$.  
\end{remark} 

\bibliographystyle{amsalpha}
\bibliography{references}
\end{document}